\documentclass[ejs]{imsart}

\RequirePackage[OT1]{fontenc}
\RequirePackage{amsthm,amsmath}
\RequirePackage[colorlinks,citecolor=blue,urlcolor=blue]{hyperref}

\usepackage{amssymb, amscd}
\usepackage{xcolor}

\arxiv{arXiv:1904.04525}

\startlocaldefs
\numberwithin{equation}{section}
\theoremstyle{plain}

\newtheorem{thm}{Theorem}[section]
\newtheorem{prop}[thm]{Proposition}
\newtheorem{lem}[thm]{Lemma}
\newtheorem{cor}[thm]{Corollary}
\theoremstyle{remark}
\newtheorem{rem}[thm]{Remark}

\newcommand{\IG}{\operatorname{IG}}
\newcommand{\Var}{\operatorname{Var}}

\newcommand{\mle}{\operatorname{mle}}

\newcommand{\map}{\operatorname{map}}
\newcommand{\mmean}{\operatorname{mean}}

\newcommand{\TV}{\operatorname{TV}}

\newcommand{\R}{\mathbb{R}}

\newcommand{\Ind}{\mathbf{1}}

\newcommand{\mL}{\mathcal{L}}

\newcommand{\mN}{\mathcal{N}}
\newcommand{\mP}{\mathcal{P}}

\newcommand{\mS}{\mathcal{S}}

\newcommand{\mX}{\mathcal{X}}

\endlocaldefs

\begin{document}

\begin{frontmatter}
\title{Bayesian variance estimation in the Gaussian sequence model with partial information on the means}
\runtitle{Bayesian variance estimation}

\begin{aug}
\author{\fnms{Gianluca} \snm{Finocchio}\thanksref{a,e1}\ead[label=e1,mark]{g.finocchio@utwente.nl}}
\and
\author{\fnms{Johannes} \snm{Schmidt-Hieber}\thanksref{a,e2}\ead[label=e2,mark]{a.j.schmidt-hieber@utwente.nl}}

\address[a]{University of Twente,
\printead{e1,e2}}

\runauthor{G. Finocchio and J. Schmidt-Hieber}

\affiliation{University of Twente}

\end{aug}

\begin{abstract}
Consider the Gaussian sequence model under the additional assumption that a fixed fraction of the means is known. We study the problem of variance estimation from a frequentist Bayesian perspective. The maximum likelihood estimator (MLE) for $\sigma^2$ is biased and inconsistent. This raises the question whether the posterior is able to correct the MLE in this case. By developing a new proving strategy that uses refined properties of the posterior distribution,  we find that the marginal posterior is inconsistent for any i.i.d. prior on the mean parameters. In particular, no assumption on the decay of the prior needs to be imposed. Surprisingly, we also find that consistency can be retained for a hierarchical prior based on Gaussian mixtures. In this case we also establish a limiting shape result and determine the limit distribution. In contrast to the classical Bernstein-von Mises theorem, the limit is non-Gaussian. We show that the Bayesian analysis leads to new statistical estimators outperforming the correctly calibrated MLE in a numerical simulation study.
\end{abstract}


\end{frontmatter}

\section{Introduction}

For given $0\leq \alpha \leq 1,$ suppose we observe $n$ independent and normally distributed random variables
\begin{align} \label{model}
X_i \sim \mathcal{N}\big( \mu_i^0 \mathbf{1}(i > n\alpha), \sigma_0^2\big), \quad i=1, \ldots, n.
\end{align}
The parameters in the model are $\mu_i^0,$ $i> n\alpha$ and $\sigma_0>0.$ The goal is to estimate the variance $\sigma_0^2$ while treating the mean vector $\mu_0 := (\mu_{\lceil n\alpha \rceil}^0,\dots,\mu_n^0)$ as nuisance. For $\alpha=0,$ we recover the Gaussian sequence model. For $\alpha >0,$ this can be viewed as the Gaussian sequence model with additional knowledge that the means of the first $\lfloor n\alpha \rfloor$ observations are known (in which case we can subtract them from the data).

One can think of model \eqref{model} as a simple prototype of a combined dataset. Using for instance different measurement devices, one often faces merged datasets collected from multiple sources. The different sources might not be of the same quality concerning the underlying parameter, see \cite{Meijer2012} for an example. An alternative viewpoint is to interpret model \eqref{model} as a sparse sequence model with known support. Since a $(1-\alpha)$-fraction of the data is perturbed, we are in the dense regime. Knowledge of the support is then crucial as otherwise there is no consistent estimator for $\sigma_0^2.$

If $n$ is even and $\alpha=1/2,$ then \eqref{model} is equivalent to the Neyman-Scott model \cite{NS48} up to a reparametrization. Model \eqref{model} is in this case equivalent to observing $U_i:=(X_{n/2+i}+X_i)$ and $V_i:=(X_{n/2+i}-X_i)$ for $i=1, \ldots, n/2.$ Since $U_i$ and $V_i$ are independent, this is thus equivalent to observing independent random variables $U_i, V_i \sim \mN( \mu_{n/2+i}^0, \widetilde \sigma_0^2),$ with $\widetilde \sigma_0^2=2\sigma_0^2.$ Estimation of $\widetilde \sigma_0^2$ in the latter model is known as Neyman-Scott problem.

Although $\sigma_0^2$ can be estimated with parametric rate based on the first $n\alpha$ observations, a striking feature of the model is that the MLE for $\sigma_0^2$ is inconsistent. In fact the MLE $\widehat \sigma^2_{\mle}$ converges to $\alpha \sigma_0^2$ therefore underestimating the true variance by the factor $\alpha.$ The reason is that the likelihood of the observations with non-zero mean significantly affects the total likelihood viewed as a function in $\sigma^2.$

We study what happens when a Bayesian approach is implemented for the estimation of the variance and whether the  posterior distribution can correct for the bias of the MLE. The Bayesian method can be viewed as a weighted likelihood method: instead of taking the parameter with the largest likelihood, the posterior puts mass on parameter sets with large likelihood. Because of this, the posterior can in some cases correct the flaws of the MLE. An example are irregular models, see \cite{ghosal1995, chernozhukov2004, RSH17}.

In the first part of the paper, we prove that whenever the nuisances are independently generated from a proper distribution, the posterior does not contract around the true variance. This shows that, for a large class of natural priors, the Bayesian method is unable to correct the MLE. In frequentist Bayes, several lower bound techniques have been developed in order to describe when Bayesian methods do not work, \cite{MR2471287, MR3059077, MR3375874, MR3486423, CM18, HRSH15}. These results can be used for instance to show that a certain decay of the prior is necessary to ensure posterior contraction. Our lower bounds are of a different flavor and do not require a condition on the tail behavior. 

Since for the non-zero means no additional structure is assumed, there is no way to say something about one mean from knowledge of all the other means. Therefore, one might be tempted to think that a correlated prior on the means cannot perform better than an i.i.d. prior and consequently must lead to an inconsistent posterior as well. Surprisingly, this is not true and we construct in the second part of the article a Gaussian mixture prior for which the posterior contracts with the parametric rate around the true variance. For this prior we derive the limit distribution in the Bernstein-von Mises sense. In contrast with the classical Bernstein-von Mises theorem, the posterior limit is non-Gaussian in the case of small means. In this case the posterior also incorporates information about the second part of the sample into the estimator and we show in a simulation study that the maximum a posteriori estimate based on the limit distribution outperforms the $\sqrt{n}$-consistent estimator that only uses the observations with zero mean. 

Estimation of the variance in model \eqref{model} can also be interpreted as a semiparametric problem. The results in this article therefore contribute to the recent efforts to understand frequentist Bayes in semiparametric models. Semiparametric Bernstein-von Mises theorems are derived under various conditions in \cite{MR1947282, MR2875753, BK12, MR3405597}. For specific priors, it has been observed that there can be a large bias in the posterior limit, see \cite{MR3021557, MR3405597, RSH17}. In all the cases studied so far, it is unclear whether the bias is due to the specific choice of prior or whether this is a fundamental limitation of the Bayesian method. To the best of our knowledge, our results show for the first time, that the posterior can be inconsistent for all natural priors. 

Related to model \eqref{model}, \cite{MR3020419} studies Bayes for variance estimation of the errors in the nonparametric regression model. It is shown that if the posterior contracts around the true regression function with rate $o(n^{-1/4}),$ the marginal posterior for the variance contracts with parametric rate around the true error variance and a Bernstein-von Mises result holds.

The article is organized as follows. In Section \ref{sec.liklh}, we discuss aspects of the problem related to the likelihood and the posterior distribution. A crucial identity for the log-posterior is derived in Section \ref{sec.log_post}. This leads then to the general negative result in Section \ref{sec.prod_priors}. The Gaussian mixture prior with parametric posterior contraction is constructed in Section \ref{sec.Gaussian_mix}. This section also contains the limiting shape result and a numerical simulation study. All proofs are deferred to the appendix.

{\em Notation:} For a vector $u=(u_1, \ldots, u_k)$, we write $\|u\|^2 = \sum_{i=1}^k u_i^2$ and $\overline{u^2}=\|u\|^2/k$ for the averages of the squares (not to be confused with the squared averages). We write $n_1=\lfloor n\alpha\rfloor$ and $n_2=n-n_1$. The probability and expectation induced by model \eqref{model} are denoted by $P_0^n$ and $E_0^n.$

\section{Likelihood and posterior}\label{sec.Like_MLE_Bpost}
\label{sec.liklh}

{\bf The MLE.} For the subsequent analysis, it is convenient to split the data vector $X=(X_1, \ldots,X_n)$ in the part with zero means $Y=(X_1, \ldots, X_{n_1})$ and the observations with non-zero means $Z=(X_{n_1+1}, \ldots, X_n)$ such that $X=(Y,Z).$ The likelihood function of the model is
\begin{equation}\label{likelihood_old}
\begin{split}
L\big(\sigma^2,\mu \big| Y,Z\big) 
& = \underbrace{\frac{1}{(2\pi\sigma^2)^{n_1/2}} e^{- \frac{\|Y\|^2}{2\sigma^2}}}_{L(\sigma^2,\mu|Y) } \underbrace{\frac{1}{(2\pi\sigma^2)^{n_2/2}} e^{- \frac{ \|Z-\mu \|^2}{2\sigma^2}}}_{L(\sigma^2,\mu |Z) } \\
&= \frac{1}{(2\pi\sigma^2)^{n/2}} e^{- \frac{\|Y\|^2 + \|Z-\mu \|^2}{2\sigma^2}}.
\end{split}
\end{equation}
Maximizing over $(\sigma^2,\mu)$ yields the MLE
\begin{equation}\label{MLE}
\begin{split}
\big(\widehat{\sigma}_{\mle}^2, \widehat \mu_{\mle}\big) = \Big(\frac{\|Y\|^2}{n}, Z\Big). 
\end{split}
\end{equation}
If only based on the subsample $Y,$ the MLE for $\sigma_0^2$ would be $\|Y\|^2/n_1$ and this converges to $\sigma_0^2$ with the parametric rate $n^{-1/2}.$ Hence $\|Y\|^2/n$ converges to $\alpha \sigma_0^2.$ The MLE for $\sigma_0^2$ is therefore inconsistent and misses the true parameter $\sigma_0^2$ by a factor $\alpha.$ It is clear that there is very little extractable information about the parameter $\sigma_0^2$ in $Z.$ A frequentist estimator can simply discard $Z$ and only use the subsample $Y.$ The MLE also does this but leads to an incorrect scaling of the estimator. 

The incorrect scaling factor of the MLE can be explained in different ways. One interpretation is that the MLE can be written as
\begin{align*}
\widehat{\sigma}_{\mle}^2 = \frac{n_1}{n} \widehat{\sigma}_{Y,\mle}^2 + \frac{n_2}{n}  \widehat{\sigma}_{Z,\mle}^2,
\end{align*}
with $\widehat{\sigma}_{Y,\mle}^2= \|Y\|^2/n_1$ the MLE based on the subsample $Y$ and $ \widehat{\sigma}_{Z,\mle}^2=0$ the MLE based on the subsample $Z.$ The fact that the overall MLE just forms a linear combination of the MLEs for the subsamples shows again that too much weight is given to $Z.$

Another explanation for the incorrect scaling of the MLE is to observe that in \eqref{likelihood_old} the likelihood based on the second subsample is $L(\sigma^2,\mu |Z) \propto  \sigma^{-n_2}$ if $\mu =\widehat \mu_{\mle}.$  If we would take the likelihood only over the first part of the sample $Y$ we would obtain the optimal estimator $\|Y\|^2/n_1,$ but since the likelihhod over the full sample is the product of the likelihood functions for $Y$ and $Z,$ an additional factor $\sigma^{-n_2}$ occurs in the overall likelihood which leads to the incorrect scaling.  More generally, we conjecture that likelihood methods do not perform well for combined datasets where one part of the data is informative about a parameter and the other part is affected by nuisance parameters.

{\bf Adjusted profile likelihood.} For the profile likelihood, we first compute the maximum likelihood estimator of the nuisance parameter for fixed $\sigma^2,$ denoted by, say $\widehat \mu_{\sigma^2},$ and then maximize
\begin{align*}
\sigma^2 \mapsto L\big(\sigma^2,\widehat \mu_{\sigma^2} \big| Y,Z\big).
\end{align*}
Obviously $\widehat \mu_{\sigma^2} = Z$ for any $\sigma^2>0$ and the profile likelihood estimator coincides with the MLE for $\sigma^2$ in the Neyman-Scott problem. If the parameter of interest and the nuisance parameters are orthogonal with respect to the Fisher information, that is, 
\begin{align}
E\Big[ \frac{\partial^2}{\partial \sigma^2 \partial \mu_j} \log L\big(\sigma^2,\mu \big| Y,Z\big) \Big] =0, \quad \text{for all } j
\label{eq.orthogonality}
\end{align}
the adjusted profile likelihood estimator \cite{MR893334, MR1064420, MR1224410} is the maximizer of 
\begin{align}
\sigma^2 \mapsto \mL(\sigma^2):=\det\big( M(\sigma^2, \widehat \mu_{\sigma^2})  \big)^{-1/2} L\big(\sigma^2,\widehat \mu_{\sigma^2} \big| Y,Z\big)
\label{eq.adjusted_likelihood}
\end{align}
for the matrix valued function
\begin{align*}
M(\sigma^2, \mu) := \Big( - \frac{\partial^2}{\partial \mu_j \partial \mu_\ell} \log L\big(\sigma^2,\mu \big| Y,Z\big) \Big)_{j,\ell}
\end{align*}
and $\det()$ the determinant. It is easy to check that \eqref{eq.orthogonality} holds for model \eqref{model}. Since $-\partial^2/(\partial \mu_j \partial \mu_\ell) \ \log L\big(\sigma^2,\mu \big|Y,Z\big) = \sigma^{-2}\mathbf{1}(j=\ell),$ the adjusted profile likelihood estimator for $\sigma^2$ coincides with the MLE for the subsample $Y,$
\begin{align*}
\widehat \sigma^2 = \frac{\|Y\|^2}{n_1}.
\end{align*}
In particular, the adjusted profile likelihood results in an unbiased $\sqrt{n}$-consistent estimator for $\sigma^2.$

{\bf The posterior distribution.} From a Bayesian perspective it is quite natural to draw $\sigma^2$ and the mean vector $\mu$ from independent distributions. Due to the orthogonality with respect to the Fisher information \eqref{eq.orthogonality}, we expect no strong interactions of $\sigma^2$ and the mean parameters in the likelihood that could be taken care of by a dependent prior. Suppose that $\mu \sim \nu$ and that the prior for $\sigma^2$ has Lebesgue density $\pi.$ The marginal posterior distribution is then given by Bayes formula
\begin{equation}\label{marginal_posterior}
\begin{split}
\pi\big(\sigma^2 \big| Y,Z \big) &= \frac{ L(\sigma^2|Y,Z) \pi(\sigma^2)}{\int_{\R_+}  L(\sigma^2|Y,Z) \pi(\sigma^2) \, d\sigma^2},
\end{split}
\end{equation}
with
\begin{equation}\label{marginal_like}
\begin{split}
L(\sigma^2|Y,Z) =  \sigma^{-n} e^{- \frac{\|Y\|^2}{2\sigma^2}} \Big(\int_{\R^n} e^{- \frac{\|Z-\mu\|^2}{2\sigma^2}}d\nu(\mu)\Big).
\end{split}
\end{equation}
In \cite{sweeting} it has been argued that by using multivariate Laplace approximation, 
\begin{align}\label{eq.sweeting_argument}
\begin{split}
L(\sigma^2|Y,Z) = \mL(\sigma^2) \nu\big( \widehat \mu_{\sigma^2} \big) \big(1+ O_P(n^{-1})\big)
= \mL(\sigma^2) \nu\big( Z \big) \big(1+ O_P(n^{-1})\big),
\end{split}
\end{align}
with $\mL(\sigma^2)$ the adjusted profile likelihood in \eqref{eq.adjusted_likelihood}. This suggests that the posterior distribution should be centered around the adjusted profile likelihood estimator $\|Y\|^2/n_1,$ therefore correcting the MLE. 

{\bf Associated sequence model with random means.} For the Gaussian sequence model with partial information \eqref{model} equipped with the product prior $\pi \otimes \nu,$  define the {\em associated sequence model with random means,} where we observe independent random variables
\begin{align}
Y_i \sim \mN(0, \sigma_0^2), \  i=1,\ldots, n_1 \ \ \text{and} \ \  Z_i |\mu \sim\mN(\mu_i,\sigma_0^2), \ i=n_1+1, \ldots, n,
\label{eq.model2}
\end{align}
with $\mu \sim \nu$ and $\nu$ known. In this model, the nuisance parameters are replaced by additional randomness. The only parameter in this model is $\sigma_0^2$ and the model is therefore parametric.  

\begin{rem}
	\label{rem.marg_lik}
	The likelihood function of model \eqref{eq.model2} is $L(\sigma^2|Y,Z).$ Model \eqref{model}  and model \eqref{eq.model2} lead therefore to the same formula for the posterior distribution of $\sigma^2$ in terms of $Y,Z.$
\end{rem}

{\bf Bayes with improper uniform prior.}  If the prior on the mean vector in the Bayes formula is chosen as the Lebesgue measure, the formula for the posterior simplifies to 
\begin{equation}\label{unif post}
\begin{split}
\pi\big(\sigma^2 \big| Y,Z \big) &\propto  \sigma^{-n_1} e^{- \frac{\|Y\|^2}{2\sigma^2}}\pi(\sigma^2).
\end{split}
\end{equation}
This is the same posterior we would get if we discarded the subsample $Z.$ It follows from the parametric Bernstein-von Mises theorem that if $\pi$ is positive and continuous in a neighbourhood of $\sigma_0^2,$ the posterior contracts around the true variance $\sigma_0^2.$ Notice that in the case of uniform prior, the Laplace approximation in \eqref{eq.sweeting_argument} is exact and does not involve any remainder terms. Obviously the Lebesgue measure is not a probability measure and the prior is improper. This raises then the question whether there are also proper priors for which the marginal posterior is consistent on the whole parameter space. We will address this problem in the next sections. 

\section{On the derivative of the log-posterior}
\label{sec.log_post}
We first derive a differential equation for the posterior. Denote by $\mu | (Z, \sigma^2)$ the posterior distribution of $\mu$ for the sample $Z,$ that is,
\begin{align}\label{mu post}
d\Pi(\mu | Z, \sigma^2) = \frac{e^{-\frac{\|Z-\mu\|^2}{2\sigma^2}} d\nu(\mu)}{\int_{\R^n} e^{-\frac{\|Z-\mu\|^2}{2\sigma^2}} d\nu(\mu)}.
\end{align}
In particular, we set
\begin{align}\label{Var term}
V\big(\mu | (Z, \sigma^2) \big) :=  \int_{\R^n} \|Z-\mu\|^2 d\Pi(\mu | Z, \sigma^2).
\end{align}
The quantity $V(\mu | (Z, \sigma^2))$ measures the spread of $\Pi(\mu | Z, \sigma^2)$ around the vector $Z.$ Recall moreover the definition of $L(\sigma^2|Y,Z)$ in \eqref{marginal_like}.

\begin{prop}\label{prop_PDE}
	The marginal posterior satisfies
	\begin{equation}\label{PDE_marg_like}
	\begin{split}
	\partial_{\sigma^2} \log \frac{\pi(\sigma^2|Y,Z)}{\pi(\sigma^2)} =  \partial_{\sigma^2} \log L(\sigma^2|Y,Z) =  \frac{\|Y\|^2+ V(\mu | (Z, \sigma^2) )}{2\sigma^4} - \frac{n}{2\sigma^2}.
	\end{split}
	\end{equation}
\end{prop} 
By Remark \ref{rem.marg_lik}, the right hand side is a closed-form expression of the score function for $\sigma^2$ in the random means model \eqref{eq.model2}. If the MLE in \eqref{eq.model2} does not lie on the boundary, the score function vanishes at the MLE. From the Bernstein-van Mises phenomenon it is conceivable that the posterior will concentrate around this MLE. For the MLE to be close to the truth $\sigma_0^2,$ the score function evaluated at $\sigma_0^2$ must be $o_P(1).$ Since $\|Y\|^2=n\alpha \sigma_0^2 +O_P(\sqrt{n}),$ this leads to the condition 
\begin{align*}
\frac{V(\mu | (Z, \sigma_0^2) )}{n} = (1-\alpha) \sigma_0^2 +o_P(1).
\end{align*}
In the next section, we derive a very general negative result. The main part of the argument is to show that the previous equality does not hold in a neighborhood of $\sigma_0^2,$ see \eqref{eq.shown}.

\section{Posterior inconsistency for product priors}
\label{sec.prod_priors}

In this section we study posterior contraction under the following condition.

{\bf Prior.} The prior on $\mu$ is independent of the prior on $\sigma^2.$ Under the prior, each component of the mean vector $\mu$ is drawn independently from a distribution $\nu$ on $\mathbb{R}.$ The prior on $\sigma^2$ has a positive and continuously differentiable Lebesgue density on $\mathbb{R}_+.$

So far, $\nu$ denoted the prior on the mean vector. By a slight abuse of language we denote the prior on the individual components also by $\nu.$ The assumptions on the prior are mild enough to account for proper priors with heavy tails and possibly no moments. 

The i.i.d. prior is the natural choice, if we believe that there is no structure in the non-zero means. From \eqref{eq.model2} it follows that the corresponding sequence model with random means is
\begin{align}
Y_i \sim \mN(0, \sigma_0^2), \  i=1,\ldots, n_1 \ \text{and} \  Z_i |\mu_i \sim\mN(\mu_i,\sigma_0^2), \ i=n_1+1, \ldots, n,
\label{eq.model2_iid}
\end{align}
with $\mu_i \sim \nu.$ For $\alpha=1/2$ and unknown $\nu,$ this model has been studied in \cite{KW56}. It is shown that the MLE for $\sigma_0^2$ and the MLE for the distribution function of the means are consistent. Since the random means model leads to the same posterior distribution as explained in Remark \ref{rem.marg_lik}, this suggests that the posterior might concentrate around the truth.

We now provide a second heuristic that leads to a different conclusion indicating that it makes a huge difference whether the distribution of the means $\nu$ is known or unknown. In the framework of \eqref{eq.model2_iid}, $\nu$ is known. If $\int u^2 d\nu(u)< \infty,$ then $\overline{\mu^2} = \int u^2 d\nu(u) + O_P(n^{-1/2})$ and $\overline{Z^2} =\overline{\mu^2}+\sigma_0^2+O_P(n^{-1/2}),$ so we have $\overline{Z^2}- \int u^2 d\nu(u)= \sigma_0^2+O_P(n^{-1/2}).$ This means that model \eqref{eq.model2_iid} carries a lot of information about $\sigma_0^2$ in the sense that $\sigma_0^2$ can be estimated with parametric rate from the subsample $Z$ only. Since the posterior only sees model \eqref{eq.model2_iid} it is therefore natural to give a lot of weight to the subsample $Z$ as well, which, from a frequentist perspective, is wrong. 

This heuristic does not say anything about heavy-tailed priors with $\int u^2 d\nu(u) =\infty.$ But even in this case, we will show that the posterior is inconsistent. The first result states that in a neighborhood of $\sigma_0^2$ the posterior is increasing extremely fast with high probability. 

\begin{prop}\label{log post bound}
	Given $\alpha <1$  and the prior above, then, for all sufficiently large $\sigma_0^2,$ there exists a mean vector $\mu_0,$ such that
	\begin{equation*}
	\begin{split}
	\lim_{n\to\infty} P_0^n\Big(\Big\{\partial_{\sigma^2}\log \pi(\sigma^2|Y,Z) \geq \sigma_0^{-2}n, \ \forall \sigma^2\in \Big[\frac{\sigma_0^2}2,2\sigma_0^2\Big]\Big\}\Big) = 1.
	\end{split}
	\end{equation*}
\end{prop}

The proof of Proposition \ref{log post bound} constructs a lower bound on $\sigma_0^2$ that is independent of $n$ and moreover guarantees that $\nu$ has sufficiently small mass outside $[-\sigma_0^2,\sigma_0^2].$ It therefore depends on the tail behavior of the prior mean distribution $\nu.$ The mean vector $\mu_0$ is subsequently chosen with all means being equal to an expression only depending on $\sigma_0^2.$ Thus the means in $\mu_0$ are uniformly bounded and independent of $n$ as well.

Suppose that almost all posterior mass is close to $\sigma_0^2.$ By the previous proposition, the posterior is increasing at least up to $2\sigma_0^2.$ Hence, there must be even more mass around $2\sigma_0^2.$ This is a contradiction and shows that the posterior does not concentrate around $\sigma_0^2.$ The proof of the next theorem is based on this argument. For this result, the means in the vector $\mu_0$ can again be chosen to be uniformly bounded.

\begin{thm}\label{posterior inconsistency}
	Given $\alpha <1$  and the prior above, then, for all sufficiently large $\sigma_0^2,$ there exists a mean vector $\mu_0$ such that
	\begin{equation*}
	\begin{split}
	\lim_{n\to\infty} E_0^n\Big[\Pi\Big( \Big| \frac{\sigma^2}{\sigma_0^2}-1\Big| \leq \frac 12 \Big|Y,Z\Big)\Big] = 0.
	\end{split}
	\end{equation*}
	Consequently, the posterior is inconsistent and assigns all its mass outside of a neighbourhood of the true variance.
\end{thm}

The posterior is therefore inferior if compared to the frequentist variance estimator $\overline{Y^2},$ which achieves the parametric rate $n^{-1/2}$ in the sense that
\begin{align*}
\sup_{\sigma_0^2>0} E_0^n\Big[ \Big| \frac{\overline{Y^2}}{\sigma_0^2} -1 \Big| \Big] \lesssim n^{-1/2}.
\end{align*}
It is remarkable that no conditions on the tail behavior of the prior distribution $\nu$ are required for Theorem \ref{posterior inconsistency}. Recall that for the improper uniform prior the posterior contracts around $\sigma_0^2.$ This shows that for distributions with heavy tailed densities, very sharp bounds are required. 

To the best of our knowledge there are no negative results in the nonparametric Bayes literature that hold for such a large class of priors. The proof strategy to establish Proposition \ref{log post bound} is based on a highly non-standard shrinkage argument that will be sketched here. By expanding the square term in (\ref{Var term}) we can lower bound (\ref{PDE_marg_like}) by
\begin{align*}
\partial_{\sigma^2} \log \pi(\sigma^2|Y,Z) &\geq \frac{\|Y\|^2}{2\sigma^4} + \frac{\|Z\|^2}{2\sigma^4} - \frac{n}{2\sigma^2} -  \frac{1}{\sigma^4}\sum_{i=1}^{n_2} V_i + O_P(1),
\end{align*}
where $V_i := |Z_i| \int |\mu_i| d\Pi(\mu | Z_i, \sigma^2).$ For $\sigma^2$ close to $\sigma_0^2$, we have
\begin{align*}
\partial_{\sigma^2} \log \pi(\sigma^2|Y,Z) &\geq \frac{n_2 \overline{\mu_0^2}}{2\sigma_0^4} -  \frac{1}{\sigma_0^4}\sum_{i=1}^{n_2} V_i + O_P(\sqrt{n}).
\end{align*}
For an improper uniform prior, one can check that $V_i \geq Z_i^2$, making the lower bound negative and useless. For a proper prior, there is a shrinkage phenomenon in the sense that for any $c>0$ there are parameters $(\mu_i^0)^2\asymp\sigma_0^2$ such that $V_i \leq cZ_i^2$, with high $P_0^n-$probability. If this is the case then
\begin{align*}
\partial_{\sigma^2} \log \pi(\sigma^2|Y,Z) &\geq \bigg(\frac{1}{2}-2c\bigg)\frac{n_2}{2\sigma_0^2} + O_P(\sqrt{n}),
\end{align*}
which yields the conclusion by choosing $c>0$ small enough.

In Proposition \ref{log post bound} we showed that the posterior overshoots the true variance $\sigma_0^2$ whenever the true means are large enough. By analyzing the Gaussian case in the next section, we see that for small means the posterior will in fact underestimate $\sigma_0^2$ and that only for a small range of mean vectors, one can hope that the posterior will be able to concentrate around the true variance. 

\section{Gaussian mixture priors}
\label{sec.Gaussian_mix}

\subsection{Gaussian priors} To illustrate our approach, we first consider an i.i.d. Gaussian prior on the mean vector
\begin{align*}
\mu_i \sim \mN(0,\theta^2), \ \text{independently.}
\end{align*}
From Theorem \ref{posterior inconsistency} we already know that the posterior will be inconsistent in this case. Nevertheless, the Gaussian assumption yields more explicit formulas and this allows us to build a hierarchical prior resulting in good posterior contraction properties. By Remark \ref{rem.marg_lik}, the marginal likelihood is the same as in the sequence model with random means \eqref{eq.model2_iid}. The marginal posterior is therefore
\begin{equation}\label{gauss_posterior}
\begin{split}
\pi\big(\sigma^2 \big| Y,Z \big) 
&\propto \sigma^{-n_1}(\theta^2+\sigma^2)^{-\frac{n_2}{2}} e^{- \frac{\|Y\|^2}{2\sigma^2}} e^{- \frac{\|Z\|^2}{2(\theta^2+\sigma^2)}} \pi(\sigma^2),
\end{split}
\end{equation}
which can also be written as the product of two inverse Gamma densities. In view of the Bernstein-von Mises phenomenon, the posterior concentrates around the MLE for parametric problems. Similarly, we can argue here that the posterior will be concentrated around the value $\widehat \sigma^2$ maximizing the likelihood part of the posterior \eqref{gauss_posterior}. By differentiation, we find $n_1 \widehat{\sigma}^2 + n_2\widehat{\sigma}^4/(\widehat{\sigma}^2 + \theta^2) =  \|Y\|^2 + \widehat{\sigma}^4\|Z\|^2/(\theta^2+\widehat{\sigma}^2)^2$ and rearranging yields
\begin{equation*}
\begin{split}
\widehat{\sigma}^2 - \overline{Y^2} &= \frac{n_2}{n_1} \bigg(\frac{\widehat{\sigma}^2}{\theta^2 + \widehat{\sigma}^2} \bigg)^2 \big[\overline{Z^2} - \theta^2 - \widehat{\sigma}^2\big].
\end{split}
\end{equation*}
This can be rewritten as 
\begin{equation}\label{MMAP_eq_Gaussian2}
\begin{split}
&\widehat{\sigma}^2 - \sigma_0^2 +O_P(n^{-1/2}) \\
&= \frac{1-\alpha}{\alpha} \big(1+O(n^{-1})\big) \bigg(\frac{\widehat{\sigma}^2}{\theta^2 + \widehat{\sigma}^2} \bigg)^2  \bigg[\sigma_0^2 - \widehat{\sigma}^2 + \overline{\mu_0^2} +O_P(n^{-1/2}) - \theta^2\bigg],
\end{split}
\end{equation}
where we set $$\overline{\mu_0^2}=\|\mu_0\|^2/n_2$$ and suppress the dependence of the $O()$ term on $\sigma_0^2$ and $\mu_0.$ Since $\theta$ is fixed, this shows that for $\widehat{\sigma}^2=\sigma_0^2 +O_P(n^{-1/2}),$ we need
\begin{align}
\overline{\mu_0^2} = \theta^2 +O_P(n^{-1/2}).
	\label{eq.theta2_condition}
\end{align}
Differently speaking, to force the maximum $\widehat \sigma^2$ to be close to $\sigma_0^2,$ the variance $\theta^2$ of the prior has to match the empirical variance $\overline{\mu_0^2}$ of the nuisance parameter. We can also deduce from \eqref{MMAP_eq_Gaussian2} that if  $|\overline{\mu_0^2}-\theta^2| \gg n^{-1/2}$ and $\theta$ is fixed, then also $|\widehat \sigma^2 -\sigma_0^2| \gg n^{-1/2}.$ More precisely, we even have that $\overline{\mu_0^2}-\theta^2 \gg n^{-1/2}$ implies $\widehat \sigma^2 -\sigma_0^2 \gg n^{-1/2}$ and $\overline{\mu_0^2}-\theta^2 \ll -n^{-1/2}$ implies $\widehat \sigma^2 -\sigma_0^2 \ll -n^{-1/2}.$ This shows that, depending on the size of $\overline{\mu_0^2}$ compared to $\theta^2,$ the posterior can either overestimate or underestimate the true variance.

If $\theta$ is allowed to vary with $n$, we can make the right hand side in \eqref{MMAP_eq_Gaussian2} arbitrarily small by letting $\theta$ tend to infinity. As $\theta^2$ is the variance of the prior, the behavior resembles then that of the uniform improper prior, which, as we already know, leads to posterior consistency. If we think of a prior as a prior belief on the parameters, then the prior should not change depending on the amount of available data and, in particular, it is unnatural that the prior becomes more vague if the sample size increases. In the next section we show that there are sample size independent mixture priors leading to parametric posterior contraction rates.

\subsection{Mixture priors}
Section \ref{sec.prod_priors} explains the posterior inconsistency for an i.i.d. prior on the nuisance. It seems unintuitive that introducing dependency on the prior of the nuisance parameter can help avoiding posterior inconsistency for $\sigma_0^2.$ Surprisingly, this is not true. In this section, we first provide some intuition why mixture priors can resolve the issues of i.i.d. priors. Afterwards, we discuss and analyze a specific prior construction.

Analyzing Gaussian priors above, \eqref{eq.theta2_condition} suggests that for any nuisance parameter vector $\mu_0,$ there exists an i.i.d. prior which seems to work. This i.i.d. prior does, however, depend on the unknown $\mu_0$ and can therefore not be chosen without knowledge of the data. Intuitively, if the posterior had the chance to see all possible i.i.d. priors on $\mu,$ instead of just one, it is conceivable that it would automatically select one that is adapted to the unknown nuisance parameter and consequently leads to posterior consistency for the parameter of interest. De Finetti's theorem \cite{MR76206} states that an exchangeable prior $\nu$ over the infinite sequence $\mu = (\mu^1,\mu^2,\dots)$ can be written as a mixture over i.i.d. priors in the sense that
\begin{align*}
	\nu(A^1\times\dots\times A^k) := \int_{\mP(\R)} Q(A^1)\cdots Q(A^k) \lambda(dQ),
\end{align*} 
with $\lambda$ a probability measure on the set of probability densities $\mP(\R)$ on $\R$. Assuming interchangeability of the integrals, the posterior \eqref{marginal_posterior} then becomes
\begin{equation*}
\begin{split}
	\pi\big(\sigma^2 \big| Y,Z \big) &\propto \pi(\sigma^2) \int_{\R^n} \frac{L(\sigma^2,\mu|Y,Z)}{L(\sigma_0^2,\mu_0|Y,Z)} \nu(\mu)d\mu, \\
	&= \pi(\sigma^2) \int_{\mP(\R)} \bigg(\int_{\R^n} \frac{L(\sigma^2,\mu|Y,Z)}{L(\sigma_0^2,\mu_0|Y,Z)}  \prod_{i=1}^n q(\mu^i)d\mu^i\bigg) \lambda(dq),
\end{split}
\end{equation*}
where $q$ denotes the probability density function of $Q.$ Let $q_0$ be the i.i.d. prior maximizing the interior integral. Suppose that this is a unique maximum and that the outer integral is determined by the behavior of the integrand in a suitable neighborhood $\mS$ of $q_0.$ This means that 
\begin{equation*}
\begin{split}
	\pi\big(\sigma^2 \big| Y,Z \big) &\propto \pi(\sigma^2) \int_{\mP(\R)} \bigg(\int_{\R^n} \frac{L(\sigma^2,\mu|Y,Z)}{L(\sigma_0^2,\mu_0|Y,Z)}  \prod_{i=1}^n q(\mu^i)d\mu^i\bigg) \lambda(dq) \\
	&\approx \pi(\sigma^2) \int_{\mS} \bigg(\int_{\R^n} \frac{L(\sigma^2,\mu|Y,Z)}{L(\sigma_0^2,\mu_0|Y,Z)}  \prod_{i=1}^n q(\mu^i)d\mu^i\bigg) \lambda(dq) \\
	&\approx \pi(\sigma^2) \bigg(\int_{\R^n} \frac{L(\sigma^2,\mu|Y,Z)}{L(\sigma_0^2,\mu_0|Y,Z)}  \prod_{i=1}^n q_0(\mu^i)d\mu^i\bigg) \int_{\mS} \lambda(dq).
\end{split}
\end{equation*}
The right hand side is the posterior density of $\sigma^2$ for i.i.d. prior $\prod_{i=1}^n q_0(\mu^i)$ on the components. 


Although this argument is only a sketch, it suggests that something might be gained by mixing over i.i.d. priors instead of just choosing one. Maximizing the marginalized likelihood in \eqref{gauss_posterior} over $\theta^2$ yields 
\begin{align}
	\theta^2 = \overline{Z^2}-\sigma^2,
	\label{eq.theta2_opt}
\end{align}
if the r.h.s. is non-negative. For this choice of $\theta^2,$  \eqref{gauss_posterior} becomes $\pi(\sigma^2|Y,Z)  \propto \sigma^{-n_1}\exp(- \|Y\|^2/(2\sigma^2)) \pi(\sigma^2).$ The posterior therefore coincides with the posterior density based on the first part of the sample only, which we know has good posterior contraction properties. 


{\bf Prior.}  In a first step generate $\theta^2 \sim \gamma$, with $\gamma$ a positive Lebesgue density on $\R_+.$ Given $\theta^2,$ each non-zero mean is drawn independently from a centered normal distribution with variance $\theta^2,$ that is, $\mu_i|\theta^2 \sim \mathcal{N}(0,\theta^2), $ $i>n_1.$

Another heuristic about the posterior properties for this prior can again be derived by making the link to the associated sequence model with random means \eqref{eq.model2}. For the prior considered here, the random means model has the form 
\begin{align}
Y_i \sim \mN(0, \sigma_0^2), \  i=1,\ldots, n_1 \ \text{and} \ Z_i |\theta^2 \sim\mN(0,\theta^2+\sigma_0^2), \ i=n_1+1, \ldots, n,
\label{eq.model2_hyper}
\end{align}
with $\theta^2 \sim \gamma.$ If $\theta^2$ were a second parameter and not generated from $\gamma,$ the variance $\sigma_0^2$ would not be identifiable if only the $Z_i$'s are observed. In model \eqref{eq.model2_hyper} we know the density $\gamma,$ but this is not enough to consistently reconstruct $\sigma_0^2$ from the subsample $Z.$ By Remark \ref{rem.marg_lik}, this model leads to the same posterior for $\sigma^2.$ The posterior should therefore realize that there is little extractable information about $\sigma_0^2$ in $Z$ and discard these observations. We will see in the limiting shape result below that this is roughly what happens.

We denote by $\ell(\sigma^2|Y)$ and $\ell(\sigma^2+\theta^2|Z)$ the log-likelihoods of the sub-samples $Y$ and $Z$ coming from model \eqref{eq.model2_hyper} with $\sigma^2$ replacing $\sigma_0^2,$ that is
\begin{align}
\begin{split} \label{eq.log-like}
	\ell(\sigma^2|Y) &= -\frac{n_1}{2}\log(2\pi\sigma^2) -\frac{n_1\overline{Y^2}}{2\sigma^2}, \\
	\ell(\sigma^2+\theta^2|Z) &= -\frac{n_2}{2}\log(2\pi(\sigma^2+\theta^2)) -\frac{n_2\overline{Z^2}}{2(\sigma^2+\theta^2)}.
\end{split}
\end{align}
The log-likelihoods appearing in \eqref{eq.log-like} can be written in terms of inverse-gamma distributions. We denote by $\IG(\gamma,\beta)$ the inverse-gamma distribution with shape $\gamma>0$ and scale $\beta>0.$ The corresponding p.d.f. is 
\begin{align}\label{eq.InvGamma}
\begin{split}
	f_{\IG(\gamma,\beta)}(x) &= \frac{\beta^\gamma}{\Gamma(\gamma)} x^{-\gamma-1} e^{-\frac{\beta}{x}},
\end{split} 
\end{align} 
where $\Gamma(\cdot)$ is the Gamma function. Rewriting the posterior, we have that

\begin{lem} \label{lemma post}
	Under the Gaussian mixture prior, the marginal posterior density has the form
	\begin{equation}\label{post}
	\begin{split}
	\pi(\sigma^2|Y,Z) &\propto f_{\IG(\gamma_1, \beta_1)}(\sigma^2) \bigg(\int_0^{+\infty} f_{\IG(\gamma_2, \beta_2)}(\sigma^2+\theta^2)\gamma(\theta^2)d\theta^2\bigg) \pi(\sigma^2),
	\end{split}
	\end{equation}
	with $\gamma_1 = n_1/2-1,$ $\beta_1 = n_1\overline{Y^2}/2$ and $\gamma_2 = n_2/2-1,$ $\beta_2=n_2\overline{Z^2}/2.$ The $\IG(\gamma_1,\beta_1)-$distribution has mode $\beta_1/(\gamma_1+1) = \overline{Y^2}$ and variance $\beta_1^2/(\gamma_1-1)^2(\gamma_1-2) = O(n^{-1}),$ whereas the $\IG(\gamma_2,\beta_2)-$distribution has mode $\beta_2/(\gamma_2+1) = \overline{Z^2}$ and variance $\beta_2^2/(\gamma_2-1)^2(\gamma_2-2) = O(n^{-1}).$
\end{lem}

Starting from Lemma \ref{lemma post}, we can develop a heuristic argument on how to recover the shape of the limit posterior distribution. We interpret the posterior $\Pi(\cdot |Y,Z)$ with density \eqref{post} as the marginalized version, over the set $\theta^2\in(0,+\infty),$ of the distribution $\widetilde{\Pi}(\cdot |Y,Z)$ whose density is given by
\begin{equation}\label{pseudo_post}
\begin{split}
	\widetilde \pi(\sigma^2, \theta^2|Y,Z) &\propto f_{\IG(\gamma_1, \beta_1)}(\sigma^2)  f_{\IG(\gamma_2, \beta_2)}(\sigma^2+\theta^2)\gamma(\theta^2) \pi(\sigma^2),
\end{split}
\end{equation}
and refer to $\widetilde{\Pi}(\cdot |Y,Z)$ as the joint posterior on $(\sigma^2,\theta^2)\in(0,+\infty)^2.$ 
The first step is double localization. Thanks to the exponential tails of the inverse Gamma distribution, the joint posterior $\widetilde{\Pi}(\cdot |Y,Z)$ asymptotically concentrates on the set $\{\sigma^2\in B_1\}\cap\{\theta^2\in B_2\},$ with $B_1$ a $O(\zeta_n)$-ball centered at $\overline{Y^2}$ and $B_2$ a $O(\zeta_n)$-ball around $0\vee(\overline{Z^2}-\overline{Y^2})$ for a sequence $\zeta_n \asymp \sqrt{\log n/n}.$ This also implies that the joint posterior \eqref{pseudo_post} is arbitrarily close, in total variation distance, to the truncated posterior distribution with density $\widetilde \pi(\sigma^2, \theta^2|Y,Z) \mathbf{1}(\{\sigma^2\in B_1\}\cap\{\theta^2\in B_2\}).$ In particular, this means that the hyperparameter $\theta^2$ concentrates on a neighborhood of the maximal value derived in \eqref{eq.theta2_opt}.

Arguing as in the classical proof of the Bernstein-von Mises theorem, we can then show that the truncated posterior distribution will asymptotically not depend on the prior and prove that the posterior given by \eqref{post} behaves asymptotically like
\begin{align}\label{post_local}
	\pi_1(\sigma^2|Y,Z) &= \mathbf{1}(\sigma^2 \in B_1) f_{\IG(\gamma_1, \beta_1)}(\sigma^2) \int_{B_2} f_{\IG(\gamma_2, \beta_2)}(\sigma^2 + \theta^2) d\theta^2.
\end{align}



Using essentially Laplace approximation, we show that the log-likelihoods $\ell(\sigma^2|Y)$ and $\ell(\sigma^2+\theta^2|Z)$ in \eqref{eq.log-like} can be always uniformly approximated by a second-order Taylor expansion around their maxima $\overline{Y^2}$ and $\overline{Z^2}-\sigma^2,$ and thus the localized posterior converges in total variation distance to a distribution with density
\begin{align}\label{post_local_Gauss}
	\pi_2(\sigma^2|Y,Z) \propto \mathbf{1}(\sigma^2 \in B_1) e^{-\frac{n_1}{4\sigma_0^4} (\sigma^2 - \overline{Y^2})^2} \int_{B_2} e^{-\frac{n_2}{4(\sigma_0^2+\overline{\mu_0^2})^2} (\theta^2 + \sigma^2 - \overline{Z^2})^2} d\theta^2,
\end{align}
whose factors are a truncated Gaussian density with mode $\overline{Y^2}$ and variance $2\sigma_0^4/n_1 = O(n^{-1})$ and the integral of a truncated Gaussian density with mode $\overline{Z^2}-\sigma^2$ and variance $2(\sigma_0^2+\overline{\mu_0^2})^2/n_2 = O(n^{-1}).$  By undoing the localization argument, we can show that the restriction to the sets $B_1$ and $B_2$ can be removed from \eqref{post_local_Gauss} and the posterior given by \eqref{post} converges in total variation distance to the posterior limit distribution
\begin{align}\label{post_limit}
\begin{split}
	\pi_\infty(\sigma^2|Y,Z) &\propto  \Ind(\sigma^2\geq 0) e^{-\frac{n_1}{4\sigma_0^4} (\sigma^2 - \overline{Y^2})^2} \bigg[1-\Phi\bigg(\frac{\sqrt{n_2}(\sigma^2 - \overline{Z^2})}{\sqrt{2}(\sigma_0^2+\overline{\mu_0^2})} \bigg)\bigg],
\end{split}
\end{align}
with $\Phi$ the c.d.f. of the standard normal distribution. Recall that $\overline{Z^2}\approx \sigma_0^2 + \overline{\mu_0^2}.$ This suggests that the term involving $\Phi$ in the posterior limit distribution should asymptotically disappear if $\overline{\mu_0^2} \gg n^{-1/2}.$ The limit of the posterior should then be the truncated Gaussian
\begin{align}\label{post_gauss}
	\widetilde{\pi}_\infty(\sigma^2|Y) \propto \Ind(\sigma^2\geq 0) \exp\Big(-\frac{n_1}{4\sigma_0^4} (\sigma^2 - \overline{Y^2})^2\Big),
\end{align}
with mode $\overline{Y^2}$ and variance $2\sigma_0^4/n_1 = O(n^{-1}).$ 

%

The next result is a formal statement of the arguments mentioned above. To pass to \eqref{post_gauss} involves an additional $\log n$-factor in the signal strength of $\overline{\mu_0^2}.$ Denote by $\|\cdot\|_{\TV}$ the total variation distance and recall that the expectation $E_0^n$ is taken with respect to model \eqref{model}. 

\begin{thm}
	\label{thm.BvM}
	Let $\Pi_\infty( \cdot |Y,Z)$ and $\widetilde{\Pi}_\infty( \cdot |Y)$ be the distributions corresponding to the densities \eqref{post_limit} and \eqref{post_gauss}, respectively. If the prior densities $\gamma, \pi:[0,\infty) \rightarrow (0,\infty)$ are positive and uniformly continuous, then, for any compact sets $K\subset (0, \infty), K'\subset (-\infty,\infty),$ and $n \rightarrow \infty,$
	\begin{align*}
	\sup_{\sigma_0^2 \in K, \mu_i^0 \in K', \forall i} \, E_0^n\Big[ \big\| \Pi( \cdot | Y,Z)- \Pi_\infty( \cdot |Y,Z) \big\|_{\TV}\Big]
	\rightarrow 0.
	\end{align*}
	Moreover, if $\inf_{\mu_i^0 \in K', \forall i}  |\mu_i^0| \gg (\log n /n)^{1/4},$ then
	\begin{align*}
	\sup_{\sigma_0^2 \in K, \mu_i^0 \in K', \forall i} \, E_0^n\Big[ \big\| \Pi( \cdot | Y,Z)- \widetilde{\Pi}_\infty( \cdot |Y) \big\|_{\TV}\Big]
	\rightarrow 0.
	\end{align*}
\end{thm}

As a corollary of the proof, posterior contraction around the true variance $\sigma_0^2$ with contraction rate $O(\sqrt{\log n/n})$ can be established. In the case of large means this is an immediate consequence of the posterior limit $\widetilde{\Pi}_\infty( \cdot |Y)$ and the parametric Bernstein-von Mises theorem. For small means  it is less obvious because of the non-standard limit of the posterior.

\begin{cor}
\label{cor.post_conc}
There exists a constant $M=M(\alpha),$ such that 
	\begin{align*}
	\sup_{\sigma_0^2 \in K, \mu_i^0 \in K', \forall i} \, E_0^n\Big[\Pi\Big( \Big| \frac{\sigma^2}{\sigma_0^2}-1\Big| \geq M\sqrt{\frac{\log n}{n}} \Big| Y,Z\Big)\Big]
	\rightarrow 0.
	\end{align*}
\end{cor}

The posterior limit distribution is closely related to the class of skew normal distributions, see \cite{Az1985,Az1996}. 
We now derive an alternative characterization of the limit distribution. From the argumentation above, the p.d.f. 
\begin{align}
	\propto \mathbf{1}(\sigma^2,\theta^2\geq 0) e^{-\frac{n_1}{4\sigma_0^4} (\sigma^2 - \overline{Y^2})^2} e^{-\frac{n_2}{4(\sigma_0^2+\overline{\mu_0^2})^2} (\theta^2 + \sigma^2 - \overline{Z^2})^2}
	\label{eq.joint_post_to_marg}
\end{align}
can be viewed as the joint posterior limit of $(\sigma^2,\theta^2).$ In particular, the posterior limit distribution is the marginal distribution with respect to $\sigma^2.$ As this is clear from the context, we do not write explicitly that the following distributions are conditional on $Y,Z,$ that is, $Y,Z$ are assumed to be fixed.

%

\begin{lem}
\label{lem.altern_repr_limit2}
Let 
\begin{align*}
	\xi \sim \mN\Big( \overline{Y^2}, \frac{2\sigma_0^4}{n_1}\Big), \quad \eta \sim \mN\Big(\overline{Z^2}, \frac{2(\sigma_0^2+\overline{\mu_0^2})^2}{n_2}\Big).
\end{align*}
be independent. The distribution with p.d.f. \eqref{eq.joint_post_to_marg} coincides with the distribution of 
\begin{align*}
	(\xi,\eta-\xi) \big | (0 \leq \xi\leq \eta).
\end{align*}
In particular, the posterior limit distribution $\Pi_\infty( \cdot |Y,Z)$ coincides with the distribution of
\begin{align*}
	\xi \big | (0 \leq \xi\leq \eta).
\end{align*}
\end{lem}

If the standard deviations of $\eta,\xi$ are small compared to the means, the posterior limit distribution essentially compares the means $\overline{Y^2}$ and $\overline{Z^2}.$ This behavior is very reasonable because if $\overline{\mu_0^2}$ is small, $\overline{Y^2}\approx \overline{Z^2}$ and the subsample $Z$ becomes informative about $\sigma^2.$

The posterior limit depends on unknown quantities. A frequentist estimator mimicking the posterior would be to estimate $\sigma^2$ from the MLE for zero means $\overline{X^2}$ in the case that the means are small. To detect whether small means are present, we can check whether $\overline{Y^2}\geq \overline{Z^2},$ which leads then to the estimator 
\begin{align*}
\widetilde \sigma^2 =
\begin{cases}
\overline{Y^2}, &\text{if} \  \overline{Y^2} < \overline{Z^2}, \\
\overline{X^2}, &\text{otherwise}. 
\end{cases}
\end{align*}

\subsection{Finite sample analysis}
We compare the estimators $\widehat \sigma_Y^2 = \overline{Y^2}$ and $\widetilde \sigma^2$ to the maximum $\widehat \sigma_{\map,\infty}^2$ and the mean $\widehat \sigma_{\mmean,\infty}^2$ of the limit density $\sigma^2 \mapsto \pi_\infty(\sigma^2|Y,Z)$ for sample sizes $n\in\{10,100,1000\}.$ As discussed above, we expect to see some differences for small means. We study the performances for $\sigma_0^2=1$ and $\mu$ the vector with all entries equal to $t/n^{1/4}$ for the values $t\in \{0, 1,2,5\}.$ Since $\widehat \sigma_Y^2$ does not depend on the means, the estimator performs equally well in all setups.  Table \ref{tab} reports the average of the squared errors and the corresponding standard errors based on $10.000$ repetitions. The rescaled MLE $\widehat \sigma_Y^2$ performs worse than any of the other estimators for small signals. Among the other estimators there is no clear 'winner'. For $t=5,$ the risk of all estimators is nearly the same. For larger values of $t,$ our simulation experiments did not show any changes compared to $t=5$ and the results are therefore omitted from the table.

\begin{table}[ht]
	\caption{Comparison of the estimators for  $(\sigma_0^2, \mu_0)=(1, (t/n^{1/4},\ldots,t/n^{1/4}))$ and $t\in \{0, 1,2,5\}.$}
	\label{tab}
	\begin{tabular}{crrrrc}
		\hline
		Estim. & \multicolumn{1}{c}{$n$} & \multicolumn{1}{c}{$0$} & \multicolumn{1}{c}{$1$} & \multicolumn{1}{c}{$2$} & \multicolumn{1}{c}{$5$} \\
		\hline
		& 10 & 0.414 ($\pm$ 8.7e-03) & 0.411 ($\pm$ 8.6e-03) & 0.386 ($\pm$ 8.2e-03) & 0.399 ($\pm$ 8.4e-03) \\ 
		$\widehat \sigma_Y^2$ & 100 & 0.040 ($\pm$ 5.9e-04) & 0.040 ($\pm$ 5.9e-04) & 0.390 ($\pm$ 5.7e-04) & \textbf{0.041} ($\pm$ 6.4e-04) \\ 
		& 1000 & 0.004 ($\pm$ 5.7e-05) & 0.004 ($\pm$ 5.6e-05) & 0.004 ($\pm$ 5.8e-05) & \textbf{0.004} ($\pm$ 5.8e-05) \\ \hline 
		
		& 10 & 0.235 ($\pm$ 3.1e-03) & 0.268 ($\pm$ 4.2e-03) & 0.336 ($\pm$ 6.2e-03) & 0.399 ($\pm$ 8.4e-03) \\
		$\widetilde \sigma^2$ & 100 & \textbf{0.028} ($\pm$ 3.8e-04) & \textbf{0.031} ($\pm$ 4.2e-04) & 0.037 ($\pm$ 5.2e-04) & \textbf{0.041} ($\pm$ 6.4e-05) \\
		& 1000 & \textbf{0.003} ($\pm$ 4.3e-05) & \textbf{0.003} ($\pm$ 4.4e-05) & 0.004 ($\pm$ 5.4e-05) & \textbf{0.004} ($\pm$ 5.8e-05) \\ \hline
		
		& 10 & 0.337 ($\pm$ 3.3e-03) & 0.330 ($\pm$ 4.6e-03) & 0.359 ($\pm$ 6.9e-03) & 0.398 ($\pm$ 8.3e-03)  \\
		$\widehat \sigma_{\map,\infty}^2$ & 100 & 0.036 ($\pm$ 4.3e-04) & 0.032 ($\pm$ 4.2e-04) & \textbf{0.034} ($\pm$ 4.7e-04) & \textbf{0.041} ($\pm$ 6.3e-04) \\
		& 1000 & \textbf{0.003} ($\pm$ 4.9e-05) & \textbf{0.003} ($\pm$ 4.5e-05) & \textbf{0.003} ($\pm$ 4.9e-05) & \textbf{0.004} ($\pm$ 5.8e-05) \\ \hline
		
		& 10 & \textbf{0.167} ($\pm$ 2.1e-03) & \textbf{0.182} ($\pm$ 3.8e-03) & \textbf{0.232} ($\pm$ 5.9e-03) & \textbf{0.283} ($\pm$ 7.0e-03)  \\
		$\widehat \sigma_{\mmean,\infty}^2$ & 100 & 0.040 ($\pm$ 4.5e-04) & 0.034 ($\pm$ 4.3e-04) & \textbf{0.034} ($\pm$ 4.7e-04) & \textbf{0.041} ($\pm$ 6.2e-04) \\
		& 1000 & 0.004 ($\pm$ 5.1e-05) & \textbf{0.003} ($\pm$ 4.6e-05) & \textbf{0.003} ($\pm$ 4.9e-05) & \textbf{0.004} ($\pm$ 5.8e-05) \\ \hline
	\end{tabular}
\end{table}

There has been a long-standing debate whether Bayesian methods perform well if interpreted as frequentist methods. Results like the complete class theorem and the Bernstein-von Mises theorem have been foundational in this regard, see \cite{LC86, GhosalvdVaart2017}. Our theory highlights another instance where Bayes leads to new estimators with good finite sample properties. The analysis moreover shows that the construction of a prior resulting in a posterior with good frequentist properties can be highly non-intuitive. 

%
%

\appendix
\section{Proofs} \label{app.proofs}

\subsection{Proofs for Section \ref{sec.log_post}}

\begin{proof} [Proof of Proposition \ref{prop_PDE}] 
	By direct computation, 
	\begin{align*}
	\partial_{\sigma^2}\log L(\sigma^2|Y,Z) &= -\frac{n}{2\sigma^2} + \frac{\|Y\|^2}{2\sigma^4} + \frac{\partial_{\sigma^2}\big(\int e^{-\frac{\|Z-\mu\|^2}{2\sigma^2}}d\nu(\mu)\big)}{\int e^{-\frac{\|Z-\mu\|^2}{2\sigma^2}}d\nu(\mu)}.
	\end{align*}
	Since
	\begin{align*}
	\partial_{\sigma^2}\Big(\int e^{-\frac{\|Z-\mu\|^2}{2\sigma^2}}d\nu(\mu)\Big) = \int \frac{\|Z-\mu\|^2}{2\sigma^4} e^{-\frac{\|Z-\mu\|^2}{2\sigma^2}}d\nu(\mu),
	\end{align*}
	we recover \eqref{PDE_marg_like}.
\end{proof}

\subsection{Proofs for Section \ref{sec.prod_priors}}

\begin{proof} [Proof of Proposition \ref{log post bound}] 
	It is enough to show that the following statements hold for sufficiently large sample size $n.$ Let $Q(u)=\nu([-u, u]^c)/\nu([-u,u]).$ Since $\nu$ is a distribution function $Q(u) \rightarrow 0$ for $u\rightarrow \infty.$ We work on $I = [\sigma_0^2/2,2\sigma_0^2],$ where $\sigma_0^2$ is chosen such that
	\begin{align}
	Q(\sigma_0) \leq \exp\Big(-48\Big(17+2e^2+\frac{24}{1-\alpha}\Big)\Big),
	\label{eq.sigma02_lb1}
	\end{align}
	and $\alpha$ denotes the fraction of known zero means in the model. Notice that
	\begin{align}
	\frac{\sigma^2}2 \leq \sigma_0^2 \leq 2\sigma^2 \quad \text{for all } \ \sigma^2\in I.
	\label{eq.sigma_equiv}
	\end{align}
	Let 
	\begin{align}\label{eq.R}
	R := \frac {\sigma_0} {\sqrt{6}} \sqrt{\log \Big(\frac{1}{Q(\sigma_0)}\Big)}.
	\end{align}
	We choose the non-zero means to be
	\begin{align}\label{eq.mu0}
	\mu_i^0 := \frac{R}{2}.
	\end{align}
	The interval $I$ is compact and the prior $\pi$ is continuous and positive on $\mathbb{R}_+,$ $\inf_{\sigma^2 \in I}  \pi(\sigma^2)>0.$ Since we also assumed that $\pi'$ is continuous, we find that 
	\begin{align*}
	\sup_{\sigma^2 \in I} \frac{\sigma_0^2|\pi'(\sigma^2)|}{n\pi(\sigma^2)} \leq 1
	\end{align*}
	for all sufficiently large $n.$ With \eqref{PDE_marg_like} and \eqref{eq.sigma_equiv}, 
	\begin{equation}\label{vn}
	\begin{split}
	\inf_{\sigma^2\in I} \partial_{\sigma^2}\log \pi(\sigma^2|Y,Z)  
	&\geq \frac{n}{\sigma_0^2} \inf_{\sigma^2\in I} \Big( \frac{\sigma_0^2V(\mu | (Z, \sigma^2) )}{2n\sigma^4} - \frac{\sigma_0^2}{2\sigma^2}-1\Big) \\
	&\geq \frac{n}{\sigma_0^2}\Big( \frac{\inf_{\sigma^2\in I} V(\mu | (Z, \sigma^2) )}{8\sigma_0^2n} -2\Big).
	\end{split}
	\end{equation}
	Using \eqref{mu post} and \eqref{Var term}, we expand $V(\mu | (Z, \sigma^2),$
	\begin{equation*}
	\begin{split}
	\frac{V(\mu | (Z, \sigma^2) )}{n} &= \frac{\|Z\|^2}{n} + \frac{1}{n} \int_{\R^n} (\|\mu\|^2 - 2Z^\top \mu) \pi(\mu | Z, \sigma^2) d\mu \\
	&= \frac{\|Z\|^2}{n} + \frac{1}{n} \sum_{i=1}^n \int_{\R} (\mu_i^2 - 2Z_i\mu_i) \pi(\mu_i | Z_i, \sigma^2) d\mu_i.
	\end{split}
	\end{equation*}
	Since the integrands in the latter display are positive for $|\mu_i| \geq 2|Z_i|,$ we can set $V_i := |Z_i| \int_{|\mu| \leq 2|Z_i|} |\mu| \pi(\mu | Z_i, \sigma^2) d\mu$ and bound 
	\begin{align*}
	\begin{split}
		\frac{V(\mu | (Z, \sigma^2) )}{n} &\geq \frac{\|Z\|^2}{n} - \frac{2}{n} \sum_{i=1}^{n_2} Z_i \int_{|\mu_i|\leq 2|Z_i|} \mu_i \pi(\mu_i | Z_i, \sigma^2) d\mu_i \\
		&\geq \frac{\|Z\|^2}{n} - \frac{2}{n} \sum_{i=1}^{n_2} V_i.
	\end{split}
	\end{align*}
	As a next step in the proof, we show
	\begin{align}\label{eq.Vi_to_show}
	\inf_{\sigma^2\in I} \frac{V(\mu | (Z, \sigma^2) )}{n}
	\geq \frac{\|Z\|^2}{2n} - \frac {16}n \Big\|Z -\frac{R}{2}\Big\|^2 - \frac{2n_2} n \sigma_0^2 e^2.
	\end{align}
	To prove this inequality, we distinguish the cases $|Z_i|>R$ and $|Z_i|\leq R,$ decomposing
	\begin{equation}\label{eq.V_decomp}
	\begin{split}
	V_i &=: |Z_i| (A_i+B_i)
	\end{split}
	\end{equation}
	with 
	\begin{align}
	\begin{split}
	A_i &:= \mathbf{1}(|Z_i|>R) \int_{|\mu|\leq 2|Z_i|} |\mu| \pi(\mu | Z_i, \sigma^2) d\mu \\
	B_i &:=\mathbf{1}(|Z_i|\leq R) \int_{|\mu|\leq 2|Z_i|} |\mu| \pi(\mu | Z_i, \sigma^2) d\mu.
	\end{split}\label{eq.A-C-def}
	\end{align}	
	For the term $A_i$ of \eqref{eq.A-C-def}, observe that $A_i \leq  2|Z_i| \Ind{(|Z_i|>R)}.$ If $|Z_i|>R,$ $|Z_i|\leq 2|Z_i|-R\leq 2 |Z_i-R/2|$ and therefore,
	\begin{align}
	|Z_i| A_i \leq  8\Big( Z_i - \frac{R}{2}\Big)^2.
	\label{eq.Bi_bd_final}
	\end{align}
	Next, we bound the term $B_i$ in \eqref{eq.A-C-def}. In the sequel, we frequently make use of the fact that $\sigma^2\in I.$ The idea is to split the domain of integration $0\leq|\mu|\leq 2|Z_i|$ into sets $|\mu|\leq\sigma_0$ and $\sigma_0<|\mu|\leq2|Z_i|.$ The contribution of the first part can be bounded by $\sigma_0.$ More work is needed for the second part. By expanding the square $(\mu-Z_i)^2$ in the exponent, the $Z_i^2$-terms in the numerator and denominator cancel against each other, as they do not depend on $\mu,$ and we have
	\begin{equation*}
	\begin{split}
	B_i &= \Ind(|Z_i|\leq R) \frac{\int_{|\mu|\leq 2|Z_i|} |\mu| e^{-\frac{(\mu-Z_i)^2}{2\sigma^2}} d\nu(\mu)}{\int e^{-\frac{(\mu-Z_i)^2}{2\sigma^2}} d\nu(\mu)} \\
	&\leq  \sigma_0 +\Ind(|Z_i|\leq R)\frac{\int_{\sigma_0 < |\mu|\leq 2R} |\mu| e^{-\frac{\mu^2}{2\sigma^2}} e^{\frac{\mu Z_i}{\sigma^2}} d\nu(\mu)}{\int e^{-\frac{\mu^2}{2\sigma^2}} e^{\frac{\mu Z_i}{\sigma^2}} d\nu(\mu)}.
	\end{split}
	\end{equation*}
We now treat numerator and denominator separately. For the numerator, the function $y\mapsto y e^{-y^2/2}$ attains its maximum at $y=1$ and is bounded by $e^{-1/2}.$ This means that $ |\mu| e^{-\frac{\mu^2}{2\sigma^2}} \leq \sigma e^{-1/2}\leq \sigma_0,$ where the last step follows from \eqref{eq.sigma_equiv}. Together with \eqref{eq.sigma_equiv}, we obtain
	\begin{align*}
	\Ind(|Z_i|\leq R)\int_{\sigma_0 < |\mu|\leq 2R} |\mu| e^{-\frac{\mu^2}{2\sigma^2}} e^{\frac{\mu Z_i}{\sigma^2}} \nu(\mu)d\mu
	&\leq \sigma_0 e^{\frac{4R^2}{\sigma_0^2}}\nu\big([-\sigma_0,\sigma_0]^c\big),
	\end{align*}
using $\mu Z_i/\sigma^2\leq 4R^2/\sigma_0^2$ to bound the exponent in the integral. To derive a lower bound of the denominator, we replace the integral over $\R$ by an integral over $[-\sigma_0,\sigma_0].$ On this interval, $e^{-\mu^2/(2\sigma^2)} \geq e^{-1}$ and $\Ind(|Z_i|\leq R) e^{\frac{\mu Z_i}{\sigma^2}} \geq e^{-R^2/\sigma^2} \geq e^{-2R^2/\sigma_0^2},$ since $\sigma_0\leq R.$ We obtain
	\begin{equation*}
	\begin{split}
	\Ind(|Z_i|\leq R) \int_\R e^{-\frac{\mu^2}{2\sigma^2}} e^{\frac{\mu Z_i}{\sigma^2}} d\nu(\mu)\geq e^{-1} e^{-\frac{2R^2}{\sigma_0^2}}\nu\big([-\sigma_0, \sigma_0]\big).
	\end{split}
	\end{equation*}
	Combining this with the upper bound for the numerator yields, with \eqref{eq.sigma02_lb1}, \eqref{eq.R} and the definition of the function $Q(u),$
	\begin{equation}\label{eq.C_bd_final}
	\begin{split}
	B_i \leq  e^{1+\frac{6R^2}{\sigma_0^2}} Q(\sigma_0) \sigma_0 =  e^{1-\log Q(\sigma_0)} Q(\sigma_0) \sigma_0 = e\sigma_0  \quad \text{for all} \ \sigma^2 \in I.
	\end{split}
	\end{equation}
	Together with \eqref{eq.Bi_bd_final} and \eqref{eq.V_decomp},
	\begin{align*}
	V_i \leq 8\Big(Z_i - \frac R2\Big)^2 + |Z_i| \sigma_0e, \quad \text{for all } \ \sigma^2 \in I.
	\end{align*}
	With $|Z_i| \sigma_0 e \leq Z_i^2/4+\sigma_0^2e^2,$ we finally obtain \eqref{eq.Vi_to_show}.
	
	In a final step of the proof, we derive, on an event with large probability, a deterministic lower bound for the right hand side in \eqref{eq.Vi_to_show}. Let $U_1, \ldots, U_{n_2}$ be independent random variables. Rewriting Chebyshev's inequality yields $P(n^{-1}\sum_{i=1}^{n_2} U_i > n^{-1}\sum_{i=1}^{n_2} (E[U_i] -\sigma_0^2)) \geq 1 - \sum_{i=1}^{n_2} \Var(U_i)/(n_2\sigma_0^2)^2.$ We aply this with $U_i =Z_i^2/2-16(Z_i-R/2)^2.$ Recall that $Z_i \sim \mN(R/2, \sigma_0^2).$ Therefore, $E_0[Z_i^2] = R^2/4+\sigma_0^2$ and $E[(Z_i-R/2)^2]=\sigma_0^2.$ For the variance, $\Var_0(Z_i^2)=R^2 \sigma_0^2 +\sigma_0^4$ and $\Var((Z_i-R/2)^2) =\sigma_0^4.$ Since by assumption $\alpha <1,$ Chebyshev's inequality yields then $P_0^n(\mathcal{A}_n)\to1$ when $n\to\infty$ for the set
	\begin{equation}\label{An}
	\begin{split}
	\mathcal{A}_n := \bigg\{\frac{\|Z\|^2}{2n} - \frac {16}n \Big\|Z -\frac{R}{2}\Big\|^2 \geq \frac{n_2}{n}\Big(\frac{R^2+ 4\sigma_0^2}{8} - 17\sigma_0^2\Big)\bigg\}.
	\end{split}
	\end{equation}
	On $\mathcal{A}_n,$ we have using \eqref{eq.R}, \eqref{eq.Vi_to_show} and $Q(\sigma_0) \leq \exp(-48(17+2e^2+24/(1-\alpha))),$
	\begin{equation}
	\begin{split}
	\inf_{\sigma^2\in I}\frac{V(\mu | (Z, \sigma^2) )}{8\sigma_0^2 n} 
	&\geq \frac{n_2}{8\sigma_0^2 n}\Big(\frac{R^2}{8} - \sigma_0^2( 17 +2e^2)\Big) 
	\geq 3.
	\end{split}
	\label{eq.shown}
	\end{equation}
	The assertion follows with \eqref{vn}.		
\end{proof}

\begin{proof} [Proof of Theorem \ref{posterior inconsistency}]
	Proposition \ref{log post bound} shows that $$\inf_{\sigma^2\in [\sigma_0^2/2, 2\sigma_0^2]} \partial_{\sigma^2} \log\pi(\sigma^2|Y,Z) \geq \frac{n}{\sigma_0^2}$$ has $P_0^n$-probability tending to one. This means that for $\sigma^2 , \widetilde \sigma^2 \in  [\sigma_0^2/2, 2\sigma_0^2],$ with $\sigma^2\leq  \widetilde \sigma^2,$ we must have $\log\pi(\sigma^2|Y,Z)\leq  \log\pi(\widetilde \sigma^2|Y,Z)- n(\widetilde \sigma^2 -\sigma^2)/\sigma_0^2.$ Exponentiating this inequality for $\widetilde \sigma^2 =\sigma^2 +\sigma_0^2/2,$ yields
	\begin{equation*}
	\begin{split}
	\Pi\Big(\sigma^2 \in \Big[\frac{\sigma_0^2}2,3\frac{\sigma_0^2}2\Big]\Big|Y,Z\Big) &= \int_{\sigma_0^2/2}^{3\sigma_0^2/2} \pi_n(\sigma^2|Y,Z) d\sigma^2 \\
	&\leq e^{-n/2} \int_{\sigma_0^2}^{2\sigma_0^2} \pi_n(\sigma^2|Y,Z) d\sigma^2
	\leq e^{-n/2}
	\end{split}
	\end{equation*}
	and this completes the proof since $|\sigma^2/\sigma_0^2-1|\leq 1/2$ is equivalent to $\sigma^2 \in [\sigma_0^2/2,3\sigma_0^2/2].$
\end{proof}

\subsection{Proofs for Section \ref{sec.Gaussian_mix}}

\begin{proof} [Proof of Lemma \ref{lemma post}]
	We can write the posterior as
	\begin{align}
		\pi(\sigma^2|Y,Z) &\propto \Ind(\sigma^2\geq0) e^{\ell(\sigma^2 |Y)} \int_0^{\infty} e^{\ell(\sigma^2+\theta^2|Z)} \gamma(\theta^2) d\theta^2 \pi(\sigma^2).
	\end{align}
	By using \eqref{eq.log-like} and \eqref{eq.InvGamma} we obtain \eqref{post}.
\end{proof}

We now prepare for the proof of the limiting shape result. From \eqref{post}, the density \eqref{pseudo_post} of the joint posterior is
\begin{equation*}
\begin{split}
	\widetilde \pi(\sigma^2, \theta^2|Y,Z) &\propto \Ind(\sigma^2\geq 0, \theta^2\geq 0)e^{\ell(\sigma^2 |Y)} e^{\ell(\sigma^2+\theta^2|Z)} \gamma(\theta^2)  \pi(\sigma^2).
\end{split}
\end{equation*}
With
\begin{align}
\zeta_n :=4\sqrt{\Big(1 + \Big(\frac{\alpha}{1-\alpha}\vee\frac{1-\alpha}{\alpha} \Big)\Big) \frac{\log n }{n_1\wedge n_2}} \wedge 1,
\label{eq.etan_def}
\end{align}
define
\begin{align}
\begin{split}
B_1&:=\Big[ \frac{\overline{Y^2}}{1+\zeta_n}, \frac{\overline{Y^2}}{1-\zeta_n}\Big], \\
B_2&:=\Big[ 0 \vee \Big( \frac{\overline{Z^2}}{1+\zeta_n}- \frac{\overline{Y^2}}{1-\zeta_n}\Big), \frac{\overline{Z^2}}{1-\zeta_n}-\frac{\overline{Y^2}}{1+\zeta_n}\Big].
\end{split}
\label{eq.Bdef}
\end{align}
It is shown below that the posterior concentrates on $\{\sigma^2 \in B_1\}$ and $\{\theta^2 \in B_2\}.$ The posterior can consequently be approximated by the distribution $\Pi_1(\cdot |Y,Z)$  defined through its density \eqref{post_local}. On the localized set $(\sigma^2, \theta^2) \in B_1 \times B_2,$ we are able to replace the log-likelihoods by a quadratic expansion. This then allows us to approximate the posterior by $\Pi_2(\cdot |Y,Z)$ which is defined as the distribution with density \eqref{post_local_Gauss}. We now state the single steps formally and provide the proofs.

\begin{prop}\label{prop.preudo_post_contr} If the prior densities $\gamma, \pi:[0,\infty) \rightarrow (0,\infty)$ are positive and uniformly continuous, then there exists a sequence of sets $(A_n)_n$ such that for any compact sets $K\subset (0, \infty), K'\subset (-\infty,\infty),$
	
	\begin{itemize}
		\item[(i)] $\lim_{n \to \infty}\sup_{\sigma_0^2 \in K, \mu_i^0 \in K', \forall i} P_0^n(A_n^c) = 0.$
		\item[(ii)] 	With $B_1, B_2$ as defined in \eqref{eq.Bdef}, we have for $n\rightarrow \infty,$
		\begin{equation*}
		\begin{split}
		\sup_{\sigma_0^2 \in K, \mu_i^0 \in K', \forall i} \, \widetilde \Pi\big(\big\{ \sigma^2 \notin B_1\big\}\cup \big\{ \theta^2 \notin B_2\big\} \ \big| Y,Z\big)\mathbf{1}\big((Y,Z) \in A_n\big) \rightarrow 0.
		\end{split}
		\end{equation*}
		\item[(iii)]	 For $n\rightarrow \infty,$
		\begin{equation*}
		\begin{split}
		\sup_{\sigma_0^2 \in K, \mu_i^0 \in K', \forall i} \, \Big\| \widetilde \Pi\big( \sigma^2 \in \cdot \big | Y,Z\big) - \Pi_1(\cdot |Y,Z) \Big\|_{\TV} \mathbf{1}\big((Y,Z) \in A_n\big)\rightarrow 0.
		\end{split}
		\end{equation*}
		\item[(iv)]	 For $n\rightarrow \infty,$
		\begin{equation*}
		\begin{split}
		\sup_{\sigma_0^2 \in K, \mu_i^0 \in K', \forall i} \, \Big\| \Pi_1( \cdot | Y,Z ) - \Pi_2(\cdot |Y,Z) \Big\|_{\TV} \mathbf{1}\big((Y,Z) \in A_n\big) \rightarrow 0.
		\end{split}
		\end{equation*}	
		\item[(v)]	 For $n\rightarrow \infty,$
		\begin{equation*}
		\begin{split}
		\sup_{\sigma_0^2 \in K, \mu_i^0 \in K', \forall i} \, \Big\| \Pi_2( \cdot | Y,Z ) - \Pi_\infty(\cdot |Y,Z) \Big\|_{\TV} \mathbf{1}\big((Y,Z) \in A_n\big) \rightarrow 0.
		\end{split}
		\end{equation*}	
		\item[(vi)]	 For $n\rightarrow \infty,$ and $\inf_{\mu_i^0\in K'} |\mu_i^0| \gg (\log n /n)^{1/4}$,
		\begin{equation*}
		\begin{split}
		\sup_{\sigma_0^2 \in K, \mu_i^0 \in K', \forall i} \, \Big\| \Pi_\infty(\cdot |Y,Z) - \widetilde{\Pi}_\infty(\cdot |Y) \Big\|_{\TV} \mathbf{1}\big((Y,Z) \in A_n\big) \rightarrow 0.
		\end{split}
		\end{equation*}	
	\end{itemize}
\end{prop}

\begin{proof}[Proof of Proposition \ref{prop.preudo_post_contr}]
	Recall the definition of $\zeta_n$ in \eqref{eq.etan_def} and set
	\begin{align}
		\delta_n := C^{-1}\zeta_n = \sqrt{2\frac{\log n }{n_1\wedge n_2}} \wedge C^{-1}, \ \text{with} \  C^2 := 16 + 16\Big(\frac{\alpha}{1-\alpha}\vee\frac{1-\alpha}{\alpha} \Big).
		\label{eq.deltan_def}
	\end{align}
	Let $\underline{\sigma}_0^2 = \inf \{\sigma_0^2 \in K\}>0.$ Define the event
	\begin{align}\label{An2}
	A_n := \Big\{ \overline{Z^2}> \frac{\overline{Y^2}}{1+\delta_n/2} \Big\} \cap \Big\{\Big|\frac{\overline{Z^2} -\overline{\mu_0^2}}{\sigma_0^2}-1\Big|+ \Big|\frac{\overline{Y^2}}{\sigma_0^2}-1 \Big| \leq \delta_n\Big\}.
	\end{align}
Since $\delta_n \leq 1/2,$ this implies in particular that on $A_n,$ $\overline{Y^2}\wedge \overline{Z^2} \geq \underline \sigma_0^2/2.$

	{\em Proof of (i):}  We simplify the notation by introducing the events
	\begin{align*}
	B_n := \Big\{ \overline{Z^2}> \frac{\overline{Y^2}}{1+\delta_n/2} \Big\},\quad D_n := \Big\{\Big|\frac{\overline{Z^2} -\overline{\mu_0^2}}{\sigma_0^2}-1\Big|+ \Big|\frac{\overline{Y^2}}{\sigma_0^2}-1 \Big| \leq \delta_n\Big\},
	\end{align*}
	so that $A_n=B_n\cap D_n$. Thus $P_0^n(A_n^c) \leq P_0^n(B_n^c) + P_0^n(D_n^c).$ We show that both $P_0^n(B_n^c)$ and $P_0^n(D_n^c)$ tend to zero uniformly over compact sets of parameters. By Chebyshev's inequality,
	\begin{align*}
	P_0^n(D_n^c) &\leq P_0^n\bigg( \Big|\frac{\overline{Z^2} -\overline{\mu_0^2}}{\sigma_0^2}-1\Big| > \frac{\delta_n}{2} \bigg) + P_0^n\bigg( \Big|\frac{\overline{Y^2}}{\sigma_0^2}-1 \Big| > \frac{\delta_n}{2} \bigg) \\ &\leq 4\frac{\Var_0\Big(\frac{\overline{Z^2} -\overline{\mu_0^2}}{\sigma_0^2}\Big) + \Var_0\Big(\frac{\overline{Y^2}}{\sigma_0^2}\Big)}{\delta_n^2}.
	\end{align*}
	Since
	\begin{align*}
	\Var_0\bigg(\frac{\overline{Z^2}-\overline{\mu_0^2}}{\sigma_0^2}\bigg) = \frac{2}{n_2}+\frac{4\overline{\mu_0^2}}{n_2\sigma_0^2},\quad \Var_0\bigg(\frac{\overline{Y^2}}{\sigma_0^2}\bigg) &= \frac{2}{n_1},
	\end{align*}
	we find
	\begin{align*}
	\sup_{\sigma_0^2 \in K, \mu_i^0 \in K', \forall i} P_0^n(D_n^c) \leq \frac{8}{n_1\delta_n^2} + \frac{8}{n_2\delta_n^2} + \frac{16H}{n_2\delta_n^2}
	\end{align*}
	with 	$H := \sup_{\sigma_0^2 \in K, \mu_i^0 \in K', \forall i} (\mu_i^0)^2/\sigma_0^2 .$ Notice that $H$ is a finite constant since $K\subset (0,\infty)$ and $K'$ are compact sets. Because $\delta_n =O(\sqrt{\log n/n}),$ the previous probability tends to zero as $n$ increases. We now bound $P_0^n(B_n^c)$. Rewriting $B_n^c,$ we obtain 
	\begin{align*}
	B_n^c = \Bigg\{\bigg(1+\frac{\delta_n}{2}\bigg)\bigg(\frac{\overline{Z^2}-\overline{\mu_0^2}}{\sigma_0^2} - 1\bigg) + 1 - \frac{\overline{Y^2}}{\sigma_0^2} \leq -\frac{\delta_n}{2} - \bigg(1+\frac{\delta_n}{2}\bigg)\frac{\overline{\mu_0^2}}{\sigma_0^2} \Bigg\},
	\end{align*}
	and again by Chebyshev's inequality
	\begin{align*}
	P_0^n(B_n^c) &\leq \frac{\big(1+\frac{\delta_n}{2}\big)^2\Var_0\Big(\frac{\overline{Z^2} -\overline{\mu_0^2}}{\sigma_0^2}-1\Big) + \Var_0\Big(1 - \frac{\overline{Y^2}}{\sigma_0^2}\Big)}{\Big(\frac{\delta_n}{2} + \big(1+\frac{\delta_n}{2}\big)\frac{\overline{\mu_0^2}}{\sigma_0^2}\Big)^2} \\
	&\leq \bigg(1+\frac{\delta_n}{2}\bigg)^2\bigg(\frac{8}{n_2\delta_n^2}+\frac{16H}{n_2\delta_n^2}\bigg) + \frac{8}{n_1\delta_n^2},
	\end{align*}
	which again tends to zero for $n\to\infty$ uniformly over $\sigma_0^2 \in K, \mu_i^0 \in K', \forall i.$
	
	{\em Proof of (ii):} We work on the event $A_n$ defined in \eqref{An2} deriving deterministic lower and upper bounds for the denominator and numerator in the Bayes formula. We start with
	\begin{align}\label{eq.mass_joint_post_B1c}
		\widetilde{\Pi}(B_1^c\times\R_+|Y,Z) &= \frac{\int_{B_1^c} e^{\ell(\sigma^2 |Y)} \int_0^\infty e^{\ell(\sigma^2+\theta^2|Z)} \gamma(\theta^2) d\theta^2 \pi(\sigma^2) d\sigma^2}{\int_0^\infty e^{\ell(\sigma^2 |Y)} \int_0^\infty e^{\ell(\sigma^2+\theta^2|Z)} \gamma(\theta^2) d\theta^2 \pi(\sigma^2) d\sigma^2},
	\end{align}
	and show that on the event $A_n$ this quantity tends to $0$ when $n$ tends to infinity. The first part of the proof provides a lower bound for the denominator. For that, we restrict $\sigma^2 \in \Sigma:= [\overline{Y^2}/(1+\delta_n), \overline{Y^2}/(1+\delta_n/2)]$ and $\theta^2 \in \Theta(\sigma^2):= [\overline{Z^2}-\sigma^2, \overline{Z^2}(1+\delta_n)-\sigma^2]\subset (0, \infty),$ where the last inclusion follows since by definition of the event $A_n$ in \eqref{An2}, $Z^2-\sigma^2 \geq Z^2-\overline{Y^2}/(1+\delta_n/2) \geq 0$. The inner integral in the denominator of \eqref{eq.mass_joint_post_B1c} can be lower bounded by
	\begin{align*}
		\int_0^{\infty} e^{\ell(\sigma^2+\theta^2|Z)} \gamma(\theta^2) d\theta^2
		&\geq \int_{\Theta(\sigma^2)} e^{\ell(\sigma^2+\theta^2|Z)} d\theta^2 \ \inf_{\theta^2 \leq \overline{Z^2}(1+\delta_n)} \gamma(\theta^2).
	\end{align*}
Thanks to the definition of $A_n$ in \eqref{An2} and $\delta_n \leq 1,$ we have $\overline{Z^2} \leq \overline{\mu_0^2} + \sigma_0^2(1+\delta_n),$ so that $\overline{Z^2}(1+\delta_n) \leq 2\overline{\mu_0^2}+4\sigma_0^2.$ We then set 
	\begin{align*}
		\underline \gamma:=  \inf_{\theta^2 \leq \sup_{\sigma_0^2 \in K, \mu_i^0 \in K', \forall i} 2\overline{\mu_0^2}+4\sigma_0^2} \gamma(\theta^2)  \leq \inf_{\theta^2 \leq \overline{Z^2}(1+\delta_n)} \gamma(\theta^2).
	\end{align*}
	Since $K,K'$ are compact sets and $\gamma$ is continuous and positive, we must have $\underline \gamma>0.$ Differentiating \eqref{eq.log-like} gives $\partial_{\theta^2} \ell(\sigma^2+\theta^2|Y)= \tfrac 12 n_2(\overline{Z^2}-\sigma^2-\theta^2)/(\sigma^2+\theta^2)^2,$ so the function $\theta^2 \mapsto \ell(\sigma^2+\theta^2 |Y)$ is decreasing on $\Theta(\sigma^2)$ for any $\sigma^2.$ As a direct consequence of \eqref{eq.log-like}, we obtain
	\begin{align}\label{eq.log-like_theta_bound}
		\ell\big( \overline{Z^2}(1+\delta_n)|Z\big) &= \ell\big(\overline{Z^2}|Z\big) + \frac{n_2}{2}\big(\delta_n/(1+\delta_n) - \log(1+\delta_n)\big).
	\end{align}
	Consequently, for any $\sigma^2 \in \Sigma,$
	\begin{align}\label{eq.denom1}
	\begin{split}
	\int_0^{\infty} e^{\ell(\sigma^2+\theta^2|Z)} \gamma(\theta^2) d\theta^2
	&\geq \underline\gamma
	\overline{Z^2} \delta_n e^{\ell( \overline{Z^2}|Z)+\frac{n_2}{2}(\delta_n/(1+\delta_n) - \log(1+\delta_n))} \\
	&\geq \frac 12 \underline\gamma
	\underline \sigma_0^2 \delta_n e^{\ell( \overline{Z^2}|Z)-\frac{n_2}{4}\delta_n^2},
	\end{split}
	\end{align}
	where the last inequality follows since $\overline{Z^2}\geq\sigma_0^2/2$ on $A_n,$ $\delta_n \leq 1,$ and $-\log(1+\delta_n) \geq -\delta_n$ for $\delta_n \leq 1.$ The right hand side does not depend on $\sigma^2$ anymore. To lower bound the first integral in the denominator of \eqref{eq.mass_joint_post_B1c} we apply a similar argument. By \eqref{eq.log-like}, $\partial_{\sigma^2} \ell(\sigma^2|Y)= n_1(\overline{Y^2}-\sigma^2)/(2\sigma^4).$ This means that the function $\sigma^2 \mapsto \ell(\sigma^2 |Y)$ is increasing on $\Sigma$ and   \eqref{eq.log-like} yields
	\begin{align*}
		\ell\big(\overline{Y^2}/(1+\delta_n)\big) &= \ell\big(\overline{Y^2} |Y\big) + \frac {n_1}{2}\big(\log(1+\delta_n)-\delta_n\big).
	\end{align*} 
On $A_n,$ $\overline{Y^2}\leq\sigma_0^2(1+\delta_n)$ and therefore $\overline{Y^2}/(1+\delta_n/2) \leq 2\sigma_0^2.$ Set
	\begin{align*}
		\underline \pi  :=\inf_{\sigma^2 \leq \sup_{\sigma_0^2\in K} 2\sigma_0^2} \pi(\sigma^2) \leq \inf_{\sigma^2 \leq \overline{Y^2}/(1+\delta_n/2)} \pi(\sigma^2),
	\end{align*}
	so that $\underline{\pi} >0$ because $K$ is a compact set and $\pi$ is continuous and positive. We bound
	\begin{align}\label{eq.denom2}
	\begin{split}
		\int_0^{\infty} e^{\ell(\sigma^2 |Y)} \pi(\sigma^2) d\sigma^2 
		&\geq \inf_{\sigma^2 \in \Sigma} \pi(\sigma^2) \frac{\delta_n}2 \overline{Y^2} e^{\ell(\overline{Y^2}/(1+\delta_n) |Y)} \\
		&\geq  \underline \pi \frac{\delta_n}2 \overline{Y^2} e^{\ell(\overline{Y^2} |Y) + \frac {n_1}{2}(\log(1+\delta_n)-\delta_n)} \\
		&\geq  \frac 14 \underline \pi \delta_n \underline \sigma_0^2 e^{\ell(\overline{Y^2} |Y) - \frac {n_1}{16}\delta_n^2},
	\end{split}	
	\end{align}  
	using that on $A_n,$ $\overline{Y^2}\geq \sigma_0^2/2$ and $\log(1+\delta_n)\geq \delta_n -\delta_n^2/8$ for $0\leq \delta_n\leq 1.$ The product of the lower bounds obtained in \eqref{eq.denom1} and \eqref{eq.denom2} is then a lower bound for the denominator of \eqref{eq.mass_joint_post_B1c}.
	
In the next step we upper bound the numerator of \eqref{eq.mass_joint_post_B1c}. Firstly, observe that $\ell(\sigma^2 +\theta^2| Z) \leq \ell(\overline{Z^2}|Z)$ and 
	\begin{align}
	\int_0^\infty e^{\ell(\sigma^2 +\theta^2| Z)} \gamma(\theta^2) d\theta^2 \leq e^{\ell(\overline{Z^2}|Z)}.
	\label{eq.d1}
	\end{align}
	Secondly, since $\sigma^2 \mapsto \ell(\sigma^2|Y)$ is increasing on $(0, \overline{Y^2}]$ and decreasing on $[\overline{Y^2}, \infty),$
	\begin{align} 
	\int_0^{\overline{Y^2}/(1+\zeta_n)} e^{\ell(\sigma^2|Y)} \pi(\sigma^2) d\sigma^2
	&\leq  e^{\ell(\overline{Y^2}/(1+\zeta_n) |Y)} \notag \\
	&=  e^{\ell(\overline{Y^2} |Y) + \frac {n_1}{2} (\log(1+\zeta_n) - \zeta_n)}  \notag  \\
	&\leq e^{\ell(\overline{Y^2} |Y) - \frac {n_1}{16}  \zeta_n^2}, \label{eq.d2} \\
	\int_{\overline{Y^2}/(1-\zeta_n)}^\infty e^{\ell(\sigma^2|Y)} \pi(\sigma^2) d\sigma^2 
	&\leq  e^{\ell(\overline{Y^2}/(1-\zeta_n) |Y)}
	=  e^{\ell(\overline{Y^2} |Y) + \frac {n_1}{2} (\log(1-\zeta_n) + \zeta_n)} \notag  \\
	&\leq e^{\ell(\overline{Y^2} |Y) - \frac {n_1}{16}  \zeta_n^2}. \notag
	\end{align}
The numerator of \eqref{eq.mass_joint_post_B1c} is upper bounded by the product of the bounds obtained in \eqref{eq.d1} and \eqref{eq.d2}. Together with the bounds on the denominator in \eqref{eq.denom1} and \eqref{eq.denom2}, and $\zeta_n=C\delta_n,$ we derive, on the event $A_n,$ the following bound for \eqref{eq.mass_joint_post_B1c}: 
	\begin{align}
		\sup_{\sigma_0^2 \in K, \mu_i^0 \in K', \forall i} \, \widetilde \Pi \big(\sigma^2 \notin B_1 \big| Y,Z \big)
		\leq \frac{16}{\underline \pi \underline \gamma \underline \sigma_0^4\delta_n^2} e^{-(C^2n_1 - 4n_2 - n_1)\delta_n^2/16} \to 0.
	\label{eq.post_est1}
	\end{align}
	The convergence to zero follows since by definition of the constant $C$ in \eqref{eq.deltan_def}, $n_1 C^2 - 4n_2 - n_1 > 4n_1$ and because of $\delta_n =O(\sqrt{\log n/n}).$
	
Along similar lines, we show now that, on the event $A_n,$ $\widetilde{\Pi}(\theta^2\notin B_2 |Y,Z)\to 0$ as $n$ tends to infinity. Since $\{\theta^2 \notin B_2\} \subset \{\sigma^2 \notin B_1\} \cup (\{\sigma^2 \in B_1\}\cap \{\theta^2 \notin B_2\}),$ and $\widetilde{\Pi}(\sigma^2\notin B_1|Y,Z)$ tends to zero by \eqref{eq.post_est1}, it is sufficient to establish convergence of
	\begin{align}\label{eq.mass_joint_post_B1B2c}
		\widetilde{\Pi}(B_1\times B_2^c|Y,Z) &= \frac{\int_{B_1} e^{\ell(\sigma^2 |Y)} \int_{B_2^c} e^{\ell(\sigma^2+\theta^2|Z)} \gamma(\theta^2) d\theta^2 \pi(\sigma^2) d\sigma^2}{\int_0^\infty e^{\ell(\sigma^2 |Y)} \int_0^\infty e^{\ell(\sigma^2+\theta^2|Z)} \gamma(\theta^2) d\theta^2 \pi(\sigma^2) d\sigma^2}
	\end{align}
	to zero. We can argue similarly as for the upper bound above using that $\ell(\sigma^2|Y)\leq \ell(\overline{Y^2}|Y).$ By following the same steps as for \eqref{eq.d1} and \eqref{eq.d2} and using that $a\mapsto \ell(a |Z)$ is increasing on $(0,\overline{Z^2}]$ and decreasing on $[\overline{Z^2}, \infty),$ the numerator in \eqref{pseudo_post} integrated over the set $\{\sigma^2 \in B_1\}\cap \{\theta^2 \notin B_2\}$ is upper bounded by  \begin{align*}
	\leq e^{\ell(\overline{Y^2}|Y)} \sup_{\sigma^2 \in B_1}
	\int_{B_2^c} e^{\ell(\sigma^2 +\theta^2|Z)} \gamma(\theta^2) d\theta^2
	\leq 2e^{\ell(\overline{Y^2}|Y)+ \ell(\overline{Z^2}|Z) - \frac {n_2}{16}  \zeta_n^2}.
	\end{align*}
	Together with the lower bounds for the denominator in \eqref{eq.denom1} and \eqref{eq.denom2}, we upper bound \eqref{eq.mass_joint_post_B1B2c}, on the event $A_n,$ by
	\begin{align}\label{eq.post_est2}
	\sup_{\sigma_0^2 \in K, \mu_i^0 \in K', \forall i} \, \widetilde{\Pi}(B_1\times B_2^c|Y,Z)
	\leq \frac{32}{\underline \pi \underline \gamma \underline \sigma_0^4\delta_n^2} e^{-(C^2n_2 - 4n_2 - n_1)\delta_n^2/16}.
	\end{align}
By definition (see \eqref{eq.deltan_def}), the constant $C^2>0$ satisfies $n_2 C^2 - 4n_2 - n_1 > 4n_2.$ Since $\delta_n =O(\sqrt{\log n/n}),$ this implies that the right hand side of \eqref{eq.post_est2} is bounded above by $\lesssim n\exp(-n_2\delta_n^2/4) \to 0,$ as $n \to \infty.$ Together with \eqref{eq.post_est1}, this completes the proof for part (ii).
	
	{\em Proof of (iii):} It is well-known that for probability measures $P,Q$ defined  on the same measurable space $\mX,$  
\begin{align}
	\|P-P(\cdot |A)\|_{\TV}\leq 2P(A^c),
	\label{eq.TV_cond_bd}
\end{align}	
see Lemma E.1 in \cite{RSH17}. With $A=B_1\cap B_2,$ $P=\widetilde \Pi(\cdot |Y,Z)$ and $\Pi_0(\cdot |Y,Z)$ the distribution with density
	\begin{equation*}
	\begin{split}
	\pi_0(\sigma^2, \theta^2|Y,Z) &= \frac{e^{\ell(\sigma^2 |Y)}  e^{\ell(\sigma^2+\theta^2|Z)}  \mathbf{1}(\sigma^2 \in B_1, \theta^2 \in B_2)}{\int_{B_1} e^{\ell(\sigma^2 |Y)} (\int_{B_2} e^{\ell(\sigma^2+\theta^2|Z)} d\theta^2) d\sigma^2},
	\end{split}
	\end{equation*}
	we have that 
	\begin{align*}
	&\sup_{\sigma_0^2 \in K, \mu_i^0 \in K', \forall i} \, 
	\Big\| \widetilde\Pi\big( \sigma^2 \in \cdot\, \big | Y,Z\big) -\Pi_0\big( \sigma^2 \in \cdot\, \big | Y,Z\big)\Big\|_{\TV}\\
	&\leq \sup_{\sigma_0^2 \in K, \mu_i^0 \in K', \forall i} \, 
	\Big\| \widetilde\Pi\big( \sigma^2 \in \cdot,  \theta^2 \in \cdot \, \big | Y,Z\big) -\Pi_0\big( \sigma^2 \in \cdot,  \theta^2 \in \cdot  \, \big | Y,Z\big)\Big\|_{\TV}
	\rightarrow 0.
	\end{align*}
By bounding the $L^1$-distance between the densities, we now show that $\Pi_0 (\sigma^2 \in \cdot |Y,Z)$ and $\Pi_1 (\sigma^2 \in \cdot |Y,Z)$ are close in total variation using the following lemma.

\begin{lem}[Lemma E.3 in \cite{RSH17}]
\label{lem.E3_lem}
If $h(\sigma^2) \propto d\Pi_0 (\sigma^2 \in \cdot |Y,Z)/d\Pi_1(\sigma^2 \in \cdot |Y,Z)$ exists and $\int |h(\sigma^2)-1|d\Pi_1(\sigma^2|Y,Z) \leq \delta$ for some $\delta \in (0,1),$ then also $$\big\|\Pi_0\big( \sigma^2 \in \cdot\, \big | Y,Z) -\Pi_1( \sigma^2 \in \cdot\, \big | Y,Z) \big\|_{\TV}\leq \frac{\delta}{1-\delta}.$$ 
\end{lem}

As $h$ is the Radon-Nikodym derivative up to a multiplicative factor, we can choose 
	\begin{align*}
	h(\sigma^2) = \frac{\pi(\sigma^2)\int_{B_2} e^{\ell(\sigma^2+\theta^2|Z)} \gamma(\theta^2)d\theta^2}{  \inf_{\widetilde \sigma^2 \in B_1, \widetilde \theta^2 \in B_2}\pi(\widetilde \sigma^2)\gamma(\widetilde \theta^2)  \int_{B_2} e^{\ell(\sigma^2+\theta^2|Z)}  d\theta^2}  \mathbf{1}(\sigma^2 \in B_1).
	\end{align*}
	Then, 
	\begin{align}
	1\leq h(\sigma^2) \leq \frac{\sup_{\sigma^2 \in B_1, \theta^2\in B_2} \pi(\sigma^2)\gamma(\theta^2)}{\inf_{\widetilde \sigma^2 \in B_1, \widetilde \theta^2 \in B_2}\pi(\widetilde \sigma^2)\gamma(\widetilde \theta^2)}.
	\label{eq.h_ul}
	\end{align}
	Using the argument above, it remains to prove that $\sup_{\sigma^2 \in B_1} |h(\sigma^2) -1|\rightarrow 0$ for $n\rightarrow \infty.$ By the definition of $A_n$ and due to $\delta_n \leq \zeta_n,$ 
	\begin{align}\label{eq.B1_embedding}
	B_1 \subseteq B_1':= [\kappa_n \sigma_0^2, \kappa_n^{-1}\sigma_0^2] \quad \text{with} \ \kappa_n:=\frac{1-\zeta_n}{1+\zeta_n} = 1-2\zeta_n +O(\zeta_n^2).
	\end{align}
	Recall that $K$ is a compact set. Since $\pi$ is positive and uniformly continuous,
	\begin{align}
	\sup_{\sigma_0^2 \in K} \ \sup_{\sigma^2, \widetilde\sigma^2 \in [\kappa_n \sigma_0^2, \kappa_n^{-1}\sigma_0^2]} \Big| \frac{\pi(\sigma^2)}{\pi(\widetilde \sigma^2)}-1 \Big| \rightarrow 0.
	\label{eq.pi_variation}
	\end{align}
	Similarly, we have on the event $A_n,$ 
	\begin{align}\label{eq.B2_embedding}
	B_2 \subseteq B_2':= \Big[\frac{\overline{\mu_0^2}}{1+\zeta_n}+\Big(\kappa_n- \frac 1{\kappa_n}\Big) \sigma_0^2, \frac{\overline{\mu_0^2}}{1-\zeta_n}+\Big(\frac 1{\kappa_n}-\kappa_n \Big) \sigma_0^2\Big].
	\end{align}
	Since $\mu_i^0\in K'$ for all $i,$ the average of the squares $\overline{\mu_0^2}$ lies in the convex hull of $K'$ and
	\begin{align*}
	\sup_{\sigma_0^2 \in K, \mu_i^0 \in K', \forall i} \ \sup_{\theta^2, \widetilde\theta^2 \in B_2'} \Big|  \frac{\gamma(\theta^2)}{\gamma(\widetilde \theta^2)} -1 \Big| \rightarrow 0.
	\end{align*}
	For real numbers $u,v,$ $uv=(u-1)(v-1)+(u-1)+(v-1)+1.$ We therefore obtain with \eqref{eq.h_ul} and \eqref{eq.pi_variation}, $\sup_{\sigma^2 \in B_1} |h(\sigma^2) -1|\rightarrow 0$ for $n\rightarrow \infty.$ This completes the proof of $(iii).$

	{\em Proof of (iv):} We use the same strategy as in the proof of part $(iii),$ applying Lemma \ref{lem.E3_lem} to 
	\begin{align*}
	h(\sigma^2) = \Ind(\sigma^2\in B_1) e^{\ell(\sigma^2|Y)-\ell(\overline{Y^2}|Y)+ \frac{n_1}{4\sigma_0^4}(\sigma^2-\overline{Y^2})^2} \frac{\int_{B_2} e^{\ell(\sigma^2+\theta^2|Z)-\ell(\overline{Z^2}|Z)} d\theta^2}{\int_{B_2} e^{-\frac{n_2}{4(\sigma_0^2+\overline{\mu_0^2})^2}(\theta^2+\sigma^2-\overline{Z^2})^2} d\theta^2},
	\end{align*}
which is a constant multiple of the likelihood ratio of $\Pi_1(\sigma^2\in\cdot|Y,Z)$ and $\Pi_2(\sigma^2\in\cdot|Y,Z).$ To verify the assumptions of Lemma \ref{lem.E3_lem}, we have to show that $\sup_{\sigma_0^2 \in K} \ |h(\sigma^2)-1| \rightarrow 0$ for $n\rightarrow\infty.$ Using again the identity $uv=(u-1)(v-1)+(u-1)+(v-1)+1$ and the fact that $|\int f/\int g -1 |\leq \sup |f/g-1|,$ we find that it is enough to prove that on the event $A_n,$
	\begin{align}
	&\sup_{\sigma_0^2 \in K} \ \sup_{\sigma^2 \in B_1} \Big | \ell(\sigma^2|Y)- \ell(\overline{Y^2}|Y)+ \frac{n_1}{4\sigma_0^4}(\sigma^2-\overline{Y^2})^2 \Big| \rightarrow 0. \label{eq.iv_to_show_1} \\
	&\sup_{\sigma_0^2 \in K,\mu_i^0\in K',\forall i} \ \sup_{\sigma^2 \in B_1, \theta^2\in B_2} \Big | \ell(\sigma^2 + \theta^2|Z)- \ell(\overline{Z^2}|Z)+ \frac{n_2(\theta^2+\sigma^2-\overline{Z^2})^2}{4(\sigma_0^2+\overline{\mu_0^2})^2} \Big| \rightarrow 0. \label{eq.iv_to_show_2}
	\end{align}
To verify \eqref{eq.iv_to_show_1}, differentiating \eqref{eq.log-like} gives
	\begin{align*}
		\partial_{\sigma^2}\ell(\sigma^2|Y) &= \frac{n_1}{2\sigma^4}(\overline{Y^2}-\sigma^2),\quad \partial_{\sigma^2}\ell(\overline{Y^2}|Y) = 0, \\
		\partial_{\sigma^2}^2\ell(\sigma^2|Y) &= \frac{n_1}{2\sigma^6}(\sigma^2-2\overline{Y^2}),\quad \partial_{\sigma^2}^2\ell(\overline{Y^2}|Y) = -\frac{n_1}{2\overline{Y^2}^2} < 0, \\
		\partial_{\sigma^2}^3\ell(\sigma^2|Y) &= \frac{n_1}{\sigma^8}(3\overline{Y^2}-\sigma^2),
	\end{align*}
	and by a third-order Taylor expansion around the maximum $\overline{Y^2},$ 
	\begin{align*}
		\ell(\sigma^2|Y) &- \ell(\overline{Y^2}|Y) \\
		&= \frac{1}{2} \partial_{\sigma^2}^2 \ell(\overline{Y^2}|Y)(\sigma^2-\overline{Y^2})^2  + \frac{1}{6} \partial_{\sigma^2}^3 \ell(s^2|Y)(\sigma^2-\overline{Y^2})^3\\
		&= -\frac{n_1}{4\overline{Y^2}^2}(\sigma^2 - \overline{Y^2})^2 + \frac{n_1}{6s^8} (3\overline{Y^2}-s^2)(\sigma^2-\overline{Y^2})^3 \\
		&= -\frac{n_1}{4\sigma_0^4}(\sigma^2 - \overline{Y^2})^2 + \frac{n_1(\overline{Y^2}+\sigma_0^2)}{4\sigma_0^4 \overline{Y^2}^2} (\overline{Y^2}-\sigma_0^2)(\sigma^2 - \overline{Y^2})^2 \\
		& \quad \ + \frac{n_1}{6s^8} (3\overline{Y^2}-s^2)(\sigma^2-\overline{Y^2})^3,
	\end{align*}
for some $s^2$ between $\sigma^2$ and $\overline{Y^2}.$ We now control the smaller order terms uniformly over $\sigma^2 \in B_1.$ Observe that also $\overline{Y^2}, s^2 \in B_1.$ With \eqref{eq.B1_embedding}, $\sup_{\sigma^2, \widetilde \sigma^2 \in B_1} |\sigma^2-\widetilde \sigma^2|=O(\zeta_n)$ and $\sigma_0^2/2 \leq \sigma^2 \leq 2\sigma_0^2$ for all $\sigma^2 \in B_1.$ Moreover, since $K\subset (0, \infty)$ is compact, $\inf {\sigma_0^2\in K}>0.$ Together this shows that 
	\begin{align*}
	&\sup_{\sigma_0^2 \in K} \ \sup_{\sigma^2 \in B_1} \Big | \ell(\sigma^2|Y)- \ell(\overline{Y^2}|Y)+ \frac{n_1}{4\sigma_0^4}(\sigma^2-\overline{Y^2})^2 \Big| = O(n_1 \zeta_n^3) \rightarrow 0,
	\end{align*}	
establishing \eqref{eq.iv_to_show_1}. To prove \eqref{eq.iv_to_show_2} we argue similarly. Differentiating \eqref{eq.log-like} gives 
	\begin{align*}
		\partial_{\theta^2}\ell(\sigma^2+\theta^2|Z) &= \frac{n_2}{2(\sigma^2+\theta^2)^2}(\overline{Z^2}-\sigma^2-\theta^2),\quad \partial_{\theta^2}\ell(\overline{Z^2}|Z) = 0, \\
		\partial_{\theta^2}^2\ell(\sigma^2+\theta^2|Z) &= \frac{n_2}{2(\sigma^2+\theta^2)^3}(\theta^2+\sigma^2-2\overline{Z^2}), \ \ \  \partial_{\theta^2}^2\ell(\overline{Z^2}|Z) = -\frac{n_2}{2\overline{Z^2}^2} < 0, \\
		\partial_{\theta^2}^3\ell(\sigma^2+\theta^2|Z) &= \frac{n_2}{(\sigma^2+\theta^2)^4}(3\overline{Z^2}-\sigma^2-\theta^2),
	\end{align*} 
	and by a third-order Taylor expansion around the maximum $\theta_*^2=\overline{Z^2}-\sigma^2$,
	\begin{align*}
		&\ell(\sigma^2+\theta^2|Z) - \ell(\overline{Z^2}|Z) \\
		&= \frac{1}{2} \partial_{\theta^2}^2\ell(\overline{Z^2}|Z)(\theta^2+\sigma^2-\overline{Z^2})^2  + \frac{1}{6} \partial_{\theta^2}^3\ell(\sigma^2+s^2|Z) (\theta^2+\sigma^2-\overline{Z^2})^3\\
		&= -\frac{n_2}{4\overline{Z^2}^2}(\theta^2+\sigma^2 - \overline{Z^2})^2 + \frac{n_2}{6(\sigma^2+s^2)^4} (3\overline{Z^2}-\sigma^2-s^2)(\theta^2+\sigma^2-\overline{Z^2})^3 \\
		&= -\frac{n_2}{4(\sigma_0^2+\overline{\mu_0^2})^2}(\theta^2 + \sigma^2 - \overline{Z^2})^2 + \frac{n_2(\overline{Z^2}+\sigma_0^2+\overline{\mu_0^2})}{4 (\sigma_0^2+\overline{\mu_0^2})^2 \overline{Z^2}^2} (\overline{Z^2}-\sigma_0^2-\overline{\mu_0^2})(\theta^2+\sigma^2 - \overline{Z^2})^2 \\
		& \quad \ + \frac{n_2}{6(\sigma^2+s^2)^4} (3\overline{Z^2}-\sigma^2-s^2)(\theta^2+\sigma^2-\overline{Z^2})^3,
	\end{align*}
	for some $s^2$ between $\theta^2$ and $\overline{Z^2}-\sigma^2.$ If $(\sigma^2,\theta^2)\in B_1 \times B_2,$ then, on $A_n,$ both $\overline{Z^2}-\sigma^2$ and $s^2$ are in $B_2'.$ With \eqref{eq.B1_embedding} and \eqref{eq.B2_embedding}, we have $\sup_{u,v\in B_2'}|u-v|=O(\zeta_n)$ and $(\sigma_0^2+\overline{\mu_0^2})/2 \leq \sigma^2+s^2 \leq  2(\sigma_0^2+\overline{\mu_0^2})$ for sufficiently large $n.$ Together with the reasoning for \eqref{eq.iv_to_show_1}, this leads to
	\begin{align*}
		\sup_{\sigma_0^2 \in K,\mu_i^0\in K',\forall i} \ \sup_{\sigma^2 \in B_1, \theta^2\in B_2} \Big | \ell(\sigma^2 + \theta^2|Z)- \ell(\overline{Z^2}|Z)+ \frac{n_2}{4(\sigma_0^2+\overline{\mu_0^2})^2}(\theta^2+\sigma^2-\overline{Z^2})^2 \Big|
	\end{align*}
being bounded by $\lesssim n\zeta_n^3$ and thus converging to zero.

	{\em Proof of (v):} 	Define $\Pi_3(\cdot |Y,Z)$ as the distribution on $(0,\infty)^2,$ with density \eqref{eq.joint_post_to_marg}, that is,
	\begin{equation*}
	\begin{split}
	\pi_3(\sigma^2, \theta^2|Y,Z) &\propto \Ind(\sigma^2\geq 0, \theta^2\geq 0) e^{-\frac{n_1}{4\sigma_0^4} (\sigma^2-\overline{Y^2})^2} e^{-\frac{n_2}{4(\sigma_0^2+\overline{\mu_0^2})^2}(\theta^2+\sigma^2-\overline{Z^2})^2}. 
	\end{split}
	\end{equation*}
	and $\widetilde \Pi_3(\cdot |Y,Z)$ as the localization of $\Pi_3(\cdot|Y,Z)$ on $B_1\times B_2,$ that is, the distribution with density
	\begin{equation*}
	\begin{split}
	\widetilde \pi_3(\sigma^2, \theta^2|Y,Z) \propto \mathbf{1}(\sigma^2 \in B_1, \theta^2 \in B_2) e^{-\frac{n_1}{4\sigma_0^4} (\sigma^2-\overline{Y^2})^2} e^{-\frac {n_2}{4}(\sigma_0^2 + \overline{\mu_0^2})^{-2}(\theta^2 +\sigma^2 -\overline{Z^2})^2}.
	\end{split}
	\end{equation*}
Here $B_1, B_2$ are as defined in \eqref{eq.Bdef}. The marginal distributions of $\widetilde \Pi_3(\cdot |Y,Z)$ and $\Pi_3(\cdot |Y,Z)$ with respect to $\sigma^2$ are $\Pi_2(\cdot |Y,Z)$ and $\Pi_\infty(\cdot |Y,Z),$ respectively. Applying \eqref{eq.TV_cond_bd} yields
	\begin{equation}\label{eq.last(v)}
	\begin{split}
	\big\| \Pi_2(\cdot |Y,Z) -\Pi_\infty(\cdot |Y,Z) \big\|_{\TV}
	&\leq \big\| \widetilde \Pi_3(\cdot |Y,Z) -\Pi_3(\cdot |Y,Z) \big\|_{\TV} \\
	&\leq 2\Pi_3\big(\big\{\sigma^2\notin B_1\big\} \cup \big\{\theta^2 \notin B_2\big\} \big|Y,Z \big). 
	\end{split}
	\end{equation}	
To prove $(v),$ it remains to show that for $n\rightarrow \infty,$
	\begin{equation}
	\begin{split}
	\sup_{\sigma_0^2 \in K, \mu_i^0 \in K', \forall i} \, \Pi_3\big(\big\{ \sigma^2 \notin B_1\big\}\cup \big\{ \theta^2 \notin B_2\big\} \big| Y,Z\big)\mathbf{1}\big((Y,Z) \in A_n\big) \rightarrow 0.
	\end{split}\label{eq.v0}
	\end{equation}
By Lemma \ref{lem.altern_repr_limit2}, it is enough to prove that on $A_n,$
	\begin{equation}
	\begin{split}
	\sup_{\sigma_0^2 \in K, \mu_i^0 \in K', \forall i} \, P\big(\xi \notin B_1\big|(0\leq \xi \leq \eta)\big)+P\big(\eta-\xi \notin B_2\big|(0\leq \xi \leq \eta)\big) \rightarrow 0,
	\end{split}\label{eq.v1}
	\end{equation}
for independent $\xi \sim \mN( \overline{Y^2}, 2\sigma_0^4/n_1), \eta \sim \mN(\overline{Z^2}, 2(\sigma_0^2+\overline{\mu_0^2})^2/n_2).$ Recall that this and all the following statements in (v) should be understood conditionally on $Y,Z.$

To bound the terms, we heavily rely on the exponential bounds for tail probabilities of Gaussian variables given by Mill's ratio \cite{G41}
	\begin{align}
	\bigg(\frac{x^2}{1+x^2}\bigg)\frac{e^{-x^2/2}}{\sqrt{2\pi}x} \leq P\Big(\mN(0,1)>x\Big) \leq \frac{e^{-x^2/2}}{\sqrt{2\pi}x} ,\quad \forall x>0.
	\label{eq.MillsR}
	\end{align}

In a first step we derive a lower bound on $P(0\leq \xi \leq \eta).$ Using that on $A_n,$ $\overline{Y^2}/(1+\delta_n/2)\leq \overline{Z^2}=E[\eta],$ the definition of $\xi,$ the symmetry properties of the $\mN(0,1)$ distribution, $\sigma_0^2/2\leq \overline{Y^2}\leq 2\sigma_0^2$ on $A_n,$ and Mill's ratio, we find
\begin{align}
\begin{split}
	P\big(0\leq \xi \leq \eta\big)
	&\geq P\Big(0\leq \xi \leq \frac{\overline{Y^2}}{1+\delta_n/2}\Big)
	P\big(\overline{Z^2} \leq \eta\big) \\
	&= \frac 12 P\bigg(\mN(0,1) \in \bigg[-\frac{\sqrt{n_1}\overline{Y^2}}{\sqrt{2}\sigma_0^2}, -\frac{\sqrt{n_1}\delta_n\overline{Y^2}}{2\sqrt{2}(1+\delta_n/2)\sigma_0^2}\bigg]\bigg) \\
		&= \frac 12 P\bigg(\mN(0,1) \in \bigg[\frac{\sqrt{n_1}\delta_n\overline{Y^2}}{2\sqrt{2}(1+\delta_n/2)\sigma_0^2},\frac{\sqrt{n_1}\overline{Y^2}}{\sqrt{2}\sigma_0^2}\bigg]\bigg) \\
		&\geq \frac 12 P\bigg(\mN(0,1) \in \bigg[\frac{\sqrt{n_1}\delta_n}{\sqrt{2}},\frac{\sqrt{n_1}}{2\sqrt{2}}\bigg]\bigg) \\
		&= P\bigg(\mN(0,1) \geq \frac{\sqrt{n_1}\delta_n}{\sqrt{2}}\bigg) - P\bigg(\mN(0,1) \geq \frac{\sqrt{n_1}}{2\sqrt{2}}\bigg) \\
		& \geq  \frac{1}{2\sqrt{\pi n_1}\delta_n} e^{-\frac{n_1\delta_n^2}{4}}- \frac{2}{\sqrt{\pi n_1}} e^{-\frac{n_1}{16}}.
	\end{split}\label{eq.v2}	
\end{align}
	where in the last inequality we used that $x^2/(1+x^2)>\frac{1}{2}$ for $x>1$.
	
We now derive an upper bound for $P(\xi \notin B_1).$ Using the definition of $\xi,$ $\zeta_n \leq 1,$ $\overline{Y^2}\geq \sigma_0^2/2,$ and Mill's ratio \eqref{eq.MillsR},
\begin{align}
	\begin{split}
	P(\xi \notin B_1)
	&=
	P\bigg( \mN(0,1) \notin \bigg[-\frac{\sqrt{n_1}\zeta_n\overline{Y^2}}{\sqrt{2}(1+\zeta_n)\sigma_0^2}, \frac{\sqrt{n_1}\zeta_n\overline{Y^2}}{\sqrt{2}(1-\zeta_n)\sigma_0^2}\bigg]\bigg) \\
	&\leq 2P\bigg( \mN(0,1)> \frac{\sqrt{n_1}\zeta_n\overline{Y^2}}{\sqrt{2}(1+\zeta_n)\sigma_0^2}\bigg) \\
	&\leq 2 P\bigg( \mN(0,1) > \frac{\sqrt{n_1}\zeta_n}{4\sqrt{2}}\bigg) \\
	&\leq \frac{8}{\sqrt{\pi n_1}\zeta_n} e^{-\frac{n_1\zeta_n^2}{64}}.
	\end{split}\label{eq.v3}
\end{align}
Next, we obtain a similar bound for $P(\eta-\xi \notin B_2, \xi \leq \eta, \xi \in B_1).$ If we define the difference of two sets $U,V$ as $U-V:=\{u-v:u\in U, v\in V\},$ then, $B_2=([\overline{Z^2}/(1+\zeta_n),\overline{Z^2}/(1-\zeta_n)]-B_1)\cap \R_+.$ On the event $\xi \leq \eta, \xi \in B_1,$ we have that $\eta \in [\overline{Z^2}/(1+\zeta_n),\overline{Z^2}/(1-\zeta_n)]$ implies that $\eta-\xi \in B_2,$ which is equivalent to saying that $\eta-\xi\notin B_2$ implies $\eta \notin [\overline{Z^2}/(1+\zeta_n),\overline{Z^2}/(1-\zeta_n)].$ On $A_n,$ $|\overline{Z^2}-\overline{\mu_0^2}-\sigma_0^2|\leq \sigma_0^2\delta_n$ by definition. Because of $\delta_n \leq 1/2,$ we obtain $\overline{Z^2}\geq (\overline{\mu_0^2}+\sigma_0^2)/2.$ Together with the symmetry properties of the normal distribution, $\zeta_n \leq 1,$ and Mill's ratio \eqref{eq.MillsR}, this yields
\begin{align}
	\begin{split}
	&P\big(\eta-\xi \notin B_2, \xi \leq \eta, \xi \in B_1\big) \\
	&\leq P\bigg( \eta \notin \bigg[\frac{\overline{Z^2}}{1+\zeta_n},\frac{\overline{Z^2}}{1-\zeta_n} \bigg]\bigg) \\
	&= P\bigg( \mN(0,1) \notin \bigg[-\frac{\sqrt{n_2}\zeta_n \overline{Z^2}}{\sqrt{2}(\overline{\mu_0^2}+\sigma_0^2)(1+\zeta_n)},\frac{\sqrt{n_2}\zeta_n\overline{Z^2}}{\sqrt{2}(\overline{\mu_0^2}+\sigma_0^2)(1-\zeta_n)} \bigg]\bigg) \\
	&\leq 2P\bigg( \mN(0,1) > \frac{\sqrt{n_2}\zeta_n }{4\sqrt{2}} \bigg) \\
	&\leq \frac{8}{\sqrt{\pi n_2}\zeta_n} e^{-\frac{n_2\zeta_n^2}{64}}.
	\end{split}\label{eq.v4}
\end{align}
To prove \eqref{eq.v1}, we bound 
\begin{align*}
	P\big(\xi \notin B_1\big|(0\leq \xi \leq \eta)\big)\leq  \frac{P(\xi \notin B_1)}{P(0\leq \xi \leq \eta)}
\end{align*}
and 
\begin{align*}
	P\big(\eta-\xi \notin B_2\big|(0\leq \xi \leq \eta)\big)
	\leq \frac{P(\eta-\xi \notin B_2, \xi \in B_1,0\leq \xi \leq \eta)+P(\xi \notin B_1)}{P(0\leq \xi \leq \eta)}.
\end{align*}
Now \eqref{eq.v1} (and therefore \eqref{eq.v0}) follow from the inequalities \eqref{eq.v2}, \eqref{eq.v3}, \eqref{eq.v4} and the definition of $\delta_n.$ This completes the proof of $(v).$
	
	{\em Proof of (vi):} Recall the definitions of the densities
	\begin{align*}
	\pi_\infty(\sigma^2|Y,Z) &\propto \mathbf{1}(\sigma^2 \geq 0) \exp\Big(-\frac{n_1}{4\sigma_0^4} (\sigma^2 - \overline{Y^2})^2\Big) \bigg( 1- \Phi\Big(\frac{\sqrt{n_2}(\sigma^2 -\overline{Z^2})}{\sqrt{2}(\sigma_0^2 +\overline{\mu_0^2})}\Big) \bigg) , \\
	\widetilde{\pi}_\infty(\sigma^2|Y) &\propto \mathbf{1}(\sigma^2 \geq 0) \exp\Big(-\frac{n_1}{4\sigma_0^4} (\sigma^2 - \overline{Y^2})^2\Big) ,
	\end{align*}
	and let
	\begin{align*}
	\pi_{\infty,B_1}(\sigma^2|Y,Z) &\propto \pi_\infty(\sigma^2|Y,Z) \mathbf{1}(\sigma^2 \in B_1), \\
	\widetilde{\pi}_{\infty,B_1}(\sigma^2|Y) &\propto \widetilde{\pi}_\infty(\sigma^2|Y) \mathbf{1}(\sigma^2 \in B_1),
	\end{align*}
	be their localised versions on $B_1$. It is enough to show that, on $A_n,$
	\begin{align}
	\sup_{\sigma_0^2 \in K, \mu_i^0 \in K', \forall i} \big\| \Pi_\infty(\cdot |Y,Z) -\Pi_{\infty,B_1}(\cdot |Y,Z) \big\|_{\TV} \xrightarrow{n\to\infty}0, \label{eq.toprove(vi)-1} \\
	\sup_{\sigma_0^2 \in K, \mu_i^0 \in K', \forall i} \big\| \widetilde{\Pi}_\infty(\cdot |Y) -\widetilde{\Pi}_{\infty,B_1}(\cdot |Y) \big\|_{\TV} \xrightarrow{n\to\infty}0, \label{eq.toprove(vi)} \\
	\sup_{\sigma_0^2 \in K, \mu_i^0 \in K', \forall i} \big\| \Pi_{\infty,B_1}(\cdot |Y,Z) -\widetilde{\Pi}_{\infty,B_1}(\cdot |Y) \big\|_{\TV} \xrightarrow{n\to\infty}0. \label{eq.toprove(vi)+1}
	\end{align}
	For \eqref{eq.toprove(vi)-1}, we apply \eqref{eq.TV_cond_bd} and the fact that $\Pi_\infty(\cdot|Y,Z)$ is the marginal distribution of $\Pi_3(\cdot|Y,Z),$ finding
	\begin{align*}
	\big\| \Pi_\infty(\cdot |Y,Z) -\Pi_{\infty,B_1}(\cdot |Y,Z) \big\|_{\TV} &\leq 2\Pi_\infty(B_1^c |Y,Z)=2\Pi_3(B_1^c |Y,Z).
	\end{align*}
In $(v)$ we proved that the right hand side converges to zero uniformly over $\sigma_0^2 \in K, \mu_i^0 \in K', \forall i.$
For \eqref{eq.toprove(vi)}, we argue similarly, using that 
	\begin{align*}
		\big\| \widetilde{\Pi}_\infty(\cdot |Y) -\widetilde{\Pi}_{\infty,B_1}(\cdot |Y) \big\|_{\TV} &\leq
		2\widetilde \Pi_\infty(B_1^c |Y)=2P(\xi \notin B_1),
	\end{align*} 
with $\xi \sim 	\mN(\overline{Y^2},2\sigma_0^4/n_1).$ Using \eqref{eq.v3}, we see that the right hand side converges to zero, uniformly over $\sigma_0^2 \in K, \mu_i^0 \in K', \forall i.$
	
For \eqref{eq.toprove(vi)+1}, we apply Lemma \ref{lem.E3_lem}. On $A_n,$ the likelihood ratio of $\Pi_{\infty,B_1}(\cdot|Y,Z)$ and $\widetilde{\Pi}_{\infty,B_1}(\cdot|Y)$ is given by
	\begin{align*}
	h(\sigma^2|Y,Z) :=\Bigg(1 - \Phi\bigg(\frac{\sqrt{n_2}(\sigma^2-\overline{Z^2})}{\sqrt{2}(\overline{\mu_0^2}+\sigma_0^2)}\bigg)\Bigg)\Ind(\sigma^2\in B_1).
	\end{align*}
On $A_n,$ 
\begin{align*}
	\sup_{\sigma^2\in B_1} \sigma^2-\overline{Z^2}=\frac{\overline{Y^2}}{1-\zeta_n}-\overline{Z^2}
	&\leq 
	\frac{\sigma_0^2(1+\delta_n)}{1-\zeta_n}-\overline{\mu_0^2}-\sigma_0^2(1-\delta_n).
\end{align*}
Uniformly over $\sigma_0^2 \in K$ and $\inf_{\mu_i^0\in K'} |\mu_i^0|^2 \gg \zeta_n,$ the right hand side can be further upper bounded by $-\overline{\mu_0^2}/2$ for sufficiently large $n.$ Thus,
	\begin{align*}
	|h(\sigma^2|Y,Z)-1| &= P\bigg(\mN(0,1) \leq \frac{\sqrt{n_2}(\sigma^2-\overline{Z^2})}{\sqrt{2}(\overline{\mu_0^2}+\sigma_0^2)}\bigg) \\
	&\leq P\bigg(\mN(0,1) \geq \frac{\sqrt{n_2}\overline{\mu_0^2}}{2\sqrt{2}(\overline{\mu_0^2}+\sigma_0^2)}\bigg).
	\end{align*}
	Since $n\overline{\mu_0^2} \gg n\zeta_n \to\infty$ for $n\to\infty,$
	\begin{align*}
	\sup_{\sigma_0^2 \in K, \mu_i^0 \in K', \forall i} P\bigg(\mN(0,1) \geq \frac{\sqrt{n_2}\overline{\mu_0^2}}{2\sqrt{2}(\overline{\mu_0^2}+\sigma_0^2)}\bigg) \xrightarrow{n\to\infty}0.
	\end{align*}
	This concludes the proof of \textit{(vi)}. 
\end{proof}

\begin{proof}[Proof of Theorem \ref{thm.BvM}]
	We insert $1=\mathbf{1}((Y,Z)\in A_n)+\mathbf{1}((Y,Z)\notin A_n)$ in the expectation. Since the total variation distance of probability measures is bounded, the result follows from Proposition \ref{prop.preudo_post_contr}. 
\end{proof}

\begin{proof}[Proof of Corollary \ref{cor.post_conc}]
Recall that the posterior is the marginal distribution of $\widetilde \Pi(\cdot |Y,Z)$ with respect to $\sigma^2.$ By Proposition \ref{prop.preudo_post_contr} (ii), we have that 
\begin{align*}
	\sup_{\sigma_0^2 \in K, \mu_i^0 \in K', \forall i} \, \Pi\bigg(\sigma^2 \notin \bigg[\frac{\overline{Y^2}}{1+\zeta_n}, \frac{\overline{Y^2}}{1-\zeta_n}\bigg] \bigg| Y,Z\bigg)\mathbf{1}\big((Y,Z) \in A_n\big) \rightarrow 0.
\end{align*}
Using that on $A_n,$ $\sigma_0^2(1-\delta_n) \leq \overline{Y^2} \leq \sigma_0^2(1+\delta_n),$ and $\delta_n=C^{-1}\zeta_n=O(\sqrt{\log n/n}),$ we obtain
\begin{align*}
	\sup_{\sigma_0^2 \in K, \mu_i^0 \in K', \forall i} \, \Pi\bigg(\bigg|\frac{\sigma^2}{\sigma_0^2}-1\bigg|\geq M\sqrt{\frac{\log n}{n}} \bigg| Y,Z\bigg)\mathbf{1}\big((Y,Z) \in A_n\big) \rightarrow 0
\end{align*}
for a constant $M=M(\alpha)$ that is chosen to be sufficiently large. The claim follows by splitting the expected posterior, inserting $1=\mathbf{1}((Y,Z)\in A_n)+\mathbf{1}((Y,Z)\notin A_n)$ in the expectation and using Proposition \ref{prop.preudo_post_contr} (i).
\end{proof}

\begin{proof}[Proof of Lemma \ref{lem.altern_repr_limit2}]
To prove the result, we derive an expression for the joint density of $(\xi,\eta-\xi) \big | (0 \leq \xi\leq \eta).$ Observe that
\begin{align*}
	P\big(\xi \leq s, \eta-\xi \leq t \big | (0 \leq \xi\leq \eta) \big)
	&= \frac{P(\xi \leq s, \eta-\xi \leq t, 0 \leq \xi\leq \eta)}{P(0 \leq \xi\leq \eta)} \\
	&\propto P\big((\eta-t)\vee 0 \leq \xi \leq \eta \wedge s\big).
\end{align*}
The right hand side is zero if $s\leq 0.$ Suppose now that $0\leq s \leq t.$ Conditioning on $\eta,$ the right hand side can be rewritten as
\begin{align*}
	=&\int_0^s P\big(0 \leq \xi \leq u\big) f_\eta(u) du
	+\int_s^t P\big(0 \leq \xi \leq s\big) f_\eta(u) du \\
	&+\int_t^{t+s} P\big(u-t \leq \xi \leq s\big) f_\eta(u) du.
\end{align*} 
Taking derivatives $\partial_s\partial_t,$ the density of $(\xi,\eta-\xi) \big | (0 \leq \xi\leq \eta)$ at point $(s,t)$ equals up to a multiplicative constant $f_\xi(s)f_\eta(t+s).$ Which completes the proof for the case $0\leq s \leq t.$ 

The case $0\leq t\leq s$ is similar and the proof for this case therefore omitted. 

Since the posterior limit distribution is the marginal over the first component of the joint distribution in \eqref{eq.joint_post_to_marg}, it must coincide with the distribution of $\xi \big | (0 \leq \xi\leq \eta).$
\end{proof}

\section*{Acknowledgment}
We are grateful to an anonymous referee and the AE for many helpful suggestions resulting in a major improvement of the article. The research has been supported by an NWO TOP grant.

\end{document}